\documentclass[11 pt]{amsart}
\usepackage{amsfonts}
\usepackage{amsmath,amscd}
\usepackage{fullpage}
\usepackage{amssymb}
\usepackage{centernot} 
\usepackage{enumerate} 
\usepackage{parskip}
\usepackage{pb-diagram} 
\usepackage{mathrsfs}
\usepackage[OT2,T1]{fontenc}
\usepackage{seqsplit}
\usepackage{tikz}
\usepackage{tikz-cd}

\usepackage{color}
\usepackage{array}
\usepackage{verbatim}
\usepackage{url}

\def\IZS{\text{\rm I$_0^*$}}

\def\II{\text{\rm II}}

\def\IV{\text{\rm IV}}

\DeclareSymbolFont{cyrletters}{OT2}{wncyr}{m}{n}
\DeclareMathSymbol{\Sha}{\mathalpha}{cyrletters}{"58}

\definecolor{refkey}{rgb}{1,1,1}
\definecolor{labelkey}{rgb}{1,1,1}
\definecolor{cite}{rgb}{0.9451,0.2706,0.4941}
\definecolor{ruri}{rgb}{0.0078,0.4022,0.8010}

\usepackage[%
bookmarks=true,bookmarksnumbered=true,%
colorlinks=true,linkcolor=ruri,citecolor=red%
 ]{hyperref}

\makeindex 

\def\F{{\rm \mathbb{F}}}
\def\Z{{\rm \mathbb{Z}}}
\def\N{{\rm \mathbb{N}}}
\def\Q{{\rm \mathbb{Q}}}

\def\G{{\rm \mathbb{G}}}
\def\Qbar{{\rm \overline{\mathbb{Q}}}}
\def\C{{\rm \mathbb{C}}}

\def\R{{\rm \mathbb{R}}}

\def\P{{\rm \mathbb{P}}}

\def\A{{\rm \mathbb{A}}}

\def\AJ{{\rm AJ}}
\def\NS{{\rm NS}}

\def\Aut{{\rm Aut}}

\def\NS{{\rm NS}}
\def\Pic{{\rm Pic}}

\def\Res{{\rm Res}}
\def\Sym{{\rm Sym}}

\def\GL{{\rm GL}}
\def\Gr{{\rm Gr}}
\def\CH{{\rm CH}}
\def\Corr{{\rm Corr}}
\def\Gal{{\rm Gal}}

\def\rk{{\rm rk}}

\DeclareMathOperator*{\ord}{ord}

\def\Hom{{\rm Hom}}
\def\End{{\rm End}}

\def\Spec{{\rm Spec}}

\numberwithin{equation}{section}

\newtheorem{theorem}{Theorem}[section]
\newtheorem{lemma}[theorem]{Lemma}

\newtheorem{remark}[theorem]{Remark}

\newtheorem{example}[theorem]{Example}
\newtheorem{conjecture}[theorem]{Conjecture}
\newtheorem{corollary}[theorem]{Corollary}
\newtheorem{proposition}[theorem]{Proposition}

\usepackage{tikz}

\newcommand{\define}[1]{{\fontfamily{cmss}\selectfont{#1}}}

\DeclareMathOperator{\sep}{sep}
\newcommand{\calO}{\mathcal{O}}
\DeclareMathOperator{\HH}{H}

\DeclareMathOperator{\sym}{sym}

\begin{document}
\setlength{\parskip}{2pt} 
\setlength{\parindent}{8pt}
\title{Ceresa cycles of bielliptic Picard curves} 
\author{Jef Laga}
\address{Department of Pure Mathematics and Mathematical Statistics, Wilberforce Road, Cambridge, CB3 0WB, UK}
\email{jeflaga@hotmail.com}
\author{Ari Shnidman}
\address{Einstein Institute of Mathematics, Hebrew University of Jerusalem, Israel} 
\email{ari.shnidman@gmail.com}

\maketitle

\begin{abstract}
We prove that the Chow class $\kappa_{\infty}(C_t)$ of the Ceresa cycle of the genus three curve $C_t \colon y^3 = x^4 + 2tx^2 + 1$ is torsion if and only if  $Q_t=( \sqrt[3]{t^2 -1},t)$ is a torsion point on the elliptic curve $y^2 = x^3 + 1$. In particular, there are infinitely many plane quartic curves over $\C$ with torsion Ceresa cycle.  Over $\overline{\Q}$, we show that the Beilinson--Bloch height of $\kappa_{\infty}(C_t)$ is proportional to the N\'eron--Tate height of $Q_t$.  Thus, the height of $\kappa_{\infty}(C_t)$ is nondegenerate and satisfies a Northcott property.  To prove all this, we show that the Chow motive that controls $\kappa_{\infty}(C_t)$ is isomorphic to $\mathfrak{h}^1$ of an appropriate  elliptic curve.
\end{abstract}

\tableofcontents
\makeatletter
\makeatother


\section{Introduction}

Let $k$ be an algebraically closed field and $C$ a smooth, projective, and connected curve over $k$ of genus $g \geq 2$ with Jacobian variety $J$.
Let $e$ be a degree-$1$ divisor on $C$ and let $\iota_e \colon C \hookrightarrow J$ be the Abel-Jacobi map based at $e$. 
We study the torsion behaviour of the Ceresa cycle 
\begin{align}\label{equation: ceresa cycle}
    \kappa_{e}(C) := [\iota_e(C)]- (-1)^*[\iota_e(C)] \in \CH_1(J)
\end{align}
in the Chow group modulo rational equivalence.
If $\kappa_{e}(C)$ is torsion, then $(2g-2)e= K_C$ in $\CH_0(C)\otimes \Q$, where $K_C$ is the canonical divisor class.
Moreover if $(2g-2)e = K_C$ in $\CH_0(C)\otimes \Q$, then the image of $\kappa_{e}(C)$ in $\CH_1(J)\otimes{\Q}$ is independent of $e$ and we denote it by $\kappa(C)$; see \S\ref{subsec: ceresa and GS cycles def} for these claims.
The class $\kappa(C)$ vanishes if and only if $\kappa_e(C)$ is torsion for some degree-$1$ divisor $e$.

While $\kappa_e(C)$ is always homologically trivial, it need not be trivial in $\CH_1(J)$. 
Indeed, if $k=\C$ then Ceresa famously showed \cite{Ceresa} that for a very general curve of genus $g\geq 3$, $\kappa_e(C)$ is of infinite order for every choice of $e$, even modulo algebraic equivalence.
On the other hand, if $C$ is hyperelliptic and $e$ is a Weierstrass point, then $\kappa_e(C) = 0$.
Beyond this, little is known about the set of genus-$g$ curves with a vanishing Ceresa cycle, though see \S\ref{subsec: previous results} for some recent results and examples.  

In this paper, we study the Ceresa cycles of bielliptic Picard curves over fields of characteristic $\neq 2,3$ \cite{LagaShnidman}. As $k$ is algebraically closed, such a curve $C$ has an affine model of the form
\begin{align}\label{equation: bielliptic Picard curve}
    y^3 = x^4+2tx^2+1,
\end{align}
for some $t \in k\setminus\{\pm 1\}$, unique up to replacing $t$ with $-t$. 
The unique point $\infty \in C(k)$ at infinity has the property that $(2g-2)\infty = 4\infty$ is canonical.
Our main result shows that the Ceresa cycles of these curves are governed by points on the single elliptic curve $\widehat{E} \colon y^2 = x^3 + 1$.

\begin{theorem}\label{thm:main} 
Let $C_t$ be the smooth projective curve with model \eqref{equation: bielliptic Picard curve}, and let $Q\in \widehat{E}(k)$ be a point with $y$-coordinate equal to $t$. 
Then $\kappa_{\infty}(C_t) \in \CH_1(J_t)$ is torsion if and only if $Q \in  \widehat{E}(k)$ is torsion.
\end{theorem}

In other words, $\kappa_{\infty}(C_t)$ is torsion if and only if $(\sqrt[3]{t^2-1},t)\in \widehat{E}(k)$ is torsion, for any choice of cube root.
See Theorem \ref{theorem: ceresa torsion statement in terms of Edelta} for a slightly different formulation of this theorem which also relates the torsion orders.
Theorem \ref{thm:main} (which is proved in \S\ref{subsec: the main theorem}) appears to be the first characterization of torsion Ceresa cycles in a positive-dimensional family of curves, aside from families where $\kappa(C)$ is identically zero. 
We deduce from this result the following corollary (see \S\ref{subsec: proofs}).

\begin{corollary}\label{cor: main infinitude}
    There are infinitely many $\Qbar$-isomorphism classes of genus $3$ curves $C$ with $\kappa(C)\neq 0$ that can be defined over $\Q$.
    In fact, $\kappa_{\infty}(C_t)$ has infinite order for all $t \in \Q\setminus \{0,\pm 1,\pm 3\}$.
\end{corollary}

The first line of Corollary \ref{cor: main infinitude} is not surprising given Ceresa's result, but it does not follow from it. Robin de Jong has given a different proof of this infinitude, using a  different family of curves  \cite{deJong}. 

Since $\widehat{E}$ has infinitely many torsion points, we also deduce from Theorem \ref{thm:main} that there are infinitely many plane quartic curves over $\C$ with {\it torsion} Ceresa cycle. As we explain below, our proof shows that there is an algebraic correspondence from $J_t$ to $\widehat{E}$ that sends $\kappa(C_t)$ to a multiple of $Q$. This allows us to deduce the following stronger statement:
\begin{theorem}\label{thm: arbitrarily large order}
There exist infinitely many $t\in \C$ such that $\kappa_{\infty}(C_t)$ is torsion.
Each such $t$ lies in $\Qbar$.
Moreover, the finite torsion orders of the classes $\kappa_{\infty}(C_t)$ are unbounded.
\end{theorem}
Note that it was only recently discovered that there exists {\it any} plane quartic curve with torsion Ceresa cycle: Qiu--Zhang \cite[\S4.1]{QiuZhang} showed that $C_0\colon y^3 = x^4+1$ has this property.

For $t \in \Qbar\setminus \{\pm 1\}$, we write $h(\kappa(C_t))$ for the Beilinson--Bloch height $\langle\kappa(C_t), \kappa(C_t)\rangle$ \cite{Beilinson, Kunnemann-heights} of $\kappa(C_t)$. 
We show that $h(\kappa(C_t))$ is proportional to the N\'eron--Tate height of $Q$ (Theorem \ref{thm: BB height = Neron-Tate height}), which implies the following nondegeneracy and Northcott properties.

\begin{theorem}\label{thm: northcott intro}
We have $h(\kappa(C_t)) = 0 $ if and only if $\kappa_{\infty}(C_t)$ is torsion. Moreover, for any $X \in \R$ and $d \in \N$, the set 
 \[\left\{ t \in \Qbar \setminus \{ \pm 1\} \colon h(\kappa(C_t)) \leq X \hspace{3mm} \mbox{and} \hspace{3mm} [\Q(t):\Q] \leq d\right\}\] is finite.   
\end{theorem}

Since the field of definition of $C_t$ is $\Q(t^2)$ (Lemma \ref{lemma: j-invariant bielliptic picard curve field of definition}), this implies that for every fixed $d\in \mathbb{N}$, there are finitely many $\Qbar$-isomorphism classes of curves $C_t$ with torsion Ceresa cycle that can be defined over a number field of degree $\leq d$.
In particular, there are finitely many curves $C_t$ with $\kappa_{\infty}(C_t)$ torsion that can be defined over $\Q$. In fact, there are  three such curves up to $\Qbar$-isomorphism: $C_0$, $C_{\pm 3}$ and $C_{\pm\sqrt{-3}}$ (Proposition \ref{prop: Q torsion ceresa cycles}). Note that $C_{\sqrt{-3}}$ has $\Q$-model $y^3 = x^4 + 6x^2-3$.

Finally, we show that for all $t$, the image of $\kappa_{\infty}(C_t)$ in the Griffiths group of homologically trivial cycles modulo algebraic equivalence is torsion, in the following strong sense:

\begin{theorem}\label{thm: torsion in griffiths intro}
      There exists an integer $N\geq 1$ such that for every bielliptic Picard curve $C$ over a field of characteristic not $2$ nor $3$, the cycle $N\cdot \kappa_{\infty}(C)\in \CH_1(J)$ is algebraically trivial.
\end{theorem}

Theorems \ref{thm: arbitrarily large order}-\ref{thm: torsion in griffiths intro} are quick corollaries of our proof of Theorem \ref{thm:main}, which we discuss below.

\subsection{Relation with modified diagonal cycles}

All results above hold verbatim for the canonical modified diagonal cycle $\Delta_{GS, \infty}(C_t) \in \CH_1(C_t^3)$ defined by Gross and Schoen \cite{GrossSchoen}; see \S\ref{sec: ceresa and GS cycles} for the definition of $\Delta_{GS,e}(C)$ and the class $\Delta_{GS}(C)\in \CH^3(C)\otimes\Q$. Indeed, $\kappa(C)$ vanishes if and only if $\Delta_{GS}(C)$ does \cite[Theorem 1.5.5]{ShouwuZhang}, and the relation $h(\Delta_{GS}(C)) = \frac32 h(\kappa(C))$ \cite[(1.1)]{dejong2021jumps} gives an analogue of Theorem \ref{thm: northcott intro} as well.

S.\ Zhang has shown a similar Northcott property for the height of $\Delta_{GS}(C)$ in families of curves over smooth projective bases \cite[Theorem 1.3.5]{ShouwuZhang}, i.e.\ compact subvarieties inside $\mathcal{M}_g$. As the family $C_t$ does not extend to a compact curve inside $\mathcal{M}_3$, Theorem \ref{thm: northcott intro} suggests that Zhang's Northcott property might very well extend to an open dense subset $U_g$ of $\mathcal{M}_g$. Indeed, several months after the first version of this paper was posted, Gao and Zhang proved that such a set $U_g$ exists \cite{GaoZhang-NorthcottCeresa}. Interestingly, in another preprint we have shown that the locus of Picard curves must be contained in the complement of $U_3$ \cite[\S1.3]{LagaShnidman-VanishingcriteriaCeresa}, so that Theorem \ref{thm: northcott intro} does not follow from \cite{GaoZhang-NorthcottCeresa}. 

\subsection{Summary of proof}
A curve $C$ of the form \eqref{equation: bielliptic Picard curve} has an order $6$ automorphism (or more canonically, a $\mu_6$-action), and the class $\kappa_{\infty}(C)$ is fixed by the induced $\mu_6$-action on $J$. On the other hand, over $\C$, the $\mu_6$-invariant part of the Hodge structure $\HH^3(J(\C), \Q)(1)$ is two-dimensional and of type $(1,0) + (0,1)$. This suggests that there should exist an algebraic correspondence from $J$ to an elliptic curve that sends $\kappa_{\infty}(C)$ to a point on this elliptic curve.
Instead of attempting to construct this correspondence explicitly, we prove this by maneuvering within the flexible category of (pure) Chow motives.

Namely, we show that there is an isomorphism of Chow motives 
\begin{align}\label{equation: intro iso of chow motives}
\mathfrak{h}^3(J)^{\mu_6} \simeq \mathfrak{h}^1(E^{\Delta})(-1),
\end{align}
where $J$ is the Jacobian of any curve of the form $C\colon y^3 = x^4+ax^2+b$, $E^\Delta$ is a certain twist of $\widehat{E}$ depending on $a$ and $b$, and the remainder of the notation will be explained later.
It is crucial for our proof that we find a universal such isomorphism, in the category of relative Chow motives over the parameter space $S = \A^2_{\Z[1/6]}\setminus \{\Delta_{a,b} =  0\}$. The elliptic scheme $\mathcal{E}^\Delta \to S$ can then be thought of as the universal $j$-invariant $0$ elliptic curve with a nonzero marked point $\mathcal{Q}$.  The main ingredients in the construction of this universal isomorphism are the Chow--Künneth decomposition for motives of abelian schemes (via Beauville's Fourier transform) \cite{Beauville-surlanneadechow, DeningerMurre}, Moonen's refinement of this decomposition in the presence of extra endomorphisms \cite{Moonen}, and our study of the Galois action on the endomorphism algebra of the universal Prym surface $\mathcal{P} \to S$ \cite{LagaShnidman}. The latter endomorphism algebra is a nonsplit quaternion algebra over $\Q$.

The Ceresa cycle of $C$ naturally lives in $\CH^2(\mathfrak{h}^3(J)^{\mu_6})$ (Lemma \ref{lem: ceresa specializations}), so the question becomes:  what is the corresponding element of $\CH^2(\frak{h}^1(E^{\Delta}(-1)) = E^{\Delta}(k) \otimes \Q$, under the isomorphism \eqref{equation: intro iso of chow motives}?  We show that it corresponds to a $\Q^{\times}$-multiple of the marked point $\mathcal{Q}$ by showing it universally over $S$, using two soft inputs.
First, the elliptic scheme $\mathcal{E}^{\Delta}\rightarrow S$ (which after $\G_m$-scaling, is essentially an elliptic surface) has Mordell--Weil group of rank 1, generated by $\mathcal{Q}$.
Second, we show by explicit computation that the universal Ceresa cycle $\kappa_{\infty}(\mathcal{C})$ has at least one infinite order specialization, using reduction mod $p$ arguments based on the methods of \cite{BlochCrelleI, BuhlerSchoenTop, EllenbergLoganSrinivasan}, see \S\ref{sec: ceresa and GS cycles}. 
Therefore $\kappa_{\infty}(\mathcal{C})$ corresponds to a nonzero section of $\mathcal{E}^\Delta$, and hence a multiple of $\mathcal{Q}$. 
Specializing at geometric points, we recover Theorem \ref{thm:main}. The construction allows us to quickly deduce the remaining results in the introduction as well.

The ``only if'' direction of Theorem \ref{thm:main} could be proved without the machinery of Chow motives, by implementing the above argument at the level of cohomology (via complex Abel-Jacobi maps and Hodge theory). However, this is not enough for the ``if'' direction since the injectivity of higher Abel-Jacobi maps for varieties over $\Qbar$ is very much open. The motivic approach therefore seems crucial for an unconditional result and is in any case more direct.

\subsection{Geometric interpretation via bigonal duality}\label{subsec:bigonal duality}

 $C_t \colon y^3 = x^4 + 2tx^2 + 1$ is a double cover of the elliptic curve $E_t \colon y^3 = x^2 + 2tx + 1$. Associated to any such double cover is the \define{bigonal dual cover} $\widehat{C}_t \to \widehat{E}_t$ constructed by Pantazis and Donagi \cite{Pantazis-Prymvarsgeodesicflow, Donagi-tetragonalconstruction}, whose Prym is dual to the Prym of $C_t \to E_t$; see \cite[\S2.7]{LagaShnidman} for a construction in the situation considered here.  In our case, the curve $\widehat{C}_t$ has model $y^3 = (x^2 + t)^2 - 1$ so that $\widehat{E}_t$ is isomorphic to $\widehat{E} \colon y^2 = x^3 + 1$, for all $t$. The branch locus of $\widehat{C}_t \to \widehat{E}_t$ consists of $\infty$ and the three points $(\sqrt[3]{t^2 - 1},t)$. Theorem \ref{thm:main} therefore says that the Ceresa cycle $\kappa_{\infty}(C_t)$ is torsion if and only if $\widehat{C}_t$ is branched along torsion points of $\widehat{E}_t$.  While this formulation is intriguing, our proof does not use the bigonal dual construction in a direct way.
 
 One might ask if this is a more general phenomenon: if $C\rightarrow E$ is a double cover of a genus $1$ curve by a genus $3$ curve, whose ramification locus contains a hyperflex point $\infty\in C(k)$, are the torsion properties of  $\kappa_{\infty}(C)$ related to the torsion properties of the branch locus of the dual cover $\widehat{C}\rightarrow \widehat{E}$?
Adam Logan has verified for us, using \cite{EllenbergLoganSrinivasan}, that out of a sampling of 99 such bielliptic curves whose dual cover is branched along $7$-torsion points, all of them have infinite order Ceresa cycle. Thus, the answer seems to be `no' in general. It would be interesting to test this hypothesis on other families of genus $3$ curves lying in the intersection (within $\mathcal{M}_3$) of the bielliptic locus and a Picard modular surface associated to an imaginary quadratic field other than $\Q(\sqrt{-3})$, but these families are non-trivial to write down.  
 In the aforementioned \cite{LagaShnidman-VanishingcriteriaCeresa}, we have generalized Theorem \ref{thm:main} to all genus three Picard curves $y^3 = f(x)$ using different methods. This suggests that the bielliptic structure which we exploit in this paper, while a very useful tool to understand and ``compute'' the Ceresa cycle, does not tell the whole story.

\subsection{Previous results}\label{subsec: previous results}

Several examples of nonhyperelliptic curves with torsion Ceresa cycle were recently found, conditional on Beilinson--Bloch type conjectures \cite{BisognoLiLittSrinivasan, GrossCeresa, Beauville, LilienfeldtShnidman}. 
Qiu and W.\ Zhang then gave some unconditional examples \cite{QiuZhang}, including one in genus 3 and one-dimensional families in genus $4$ and $5$.
Laterveer also found a two-dimensional family in genus $5$ \cite{Laterveer}.
Subsequently, Qiu--Zhang  found examples of plane quartic curves (dominated by Shimura curves) with trivial automorphism group and torsion Ceresa cycle \cite{QiuZhangII}.

The proofs of \cite{QiuZhang} and \cite{Laterveer} rely on the vanishing of the Chow motive $(\frak{h}^1(C)^{\otimes 3})^{\Aut(C)}$, and hence of $\frak{h}^3(J)^{\Aut(C)}$, for certain curves $C$, generalizing the observation that hyperelliptic curves have trivial Ceresa cycle.
For our curves $C = C_t$ with $t\neq 0,\pm 1$, the motive $\frak{h}^3(J)^{\Aut(C)} = \frak{h}^3(J)^{\mu_6}$ does not vanish, but \eqref{equation: intro iso of chow motives} shows that it is a direct summand of the motive of a curve, hence is still reasonably accessible. It would be interesting to study other families of curves where $\frak{h}^3(J)^{\Aut(C)}$ is ``small'' in a suitable sense. More generally, one can try to study the nonzero isotypic components of $\kappa(C)$ with respect to the action of $\End(J)$ on $\frak{h}^3(J)$.  For example, even though some of the examples in \cite{QiuZhangII} have trivial automorphism group, they all satisfy $\End(J) \neq \Z$.  

Aside from the aforementioned \cite{deJong}, Eskandari and Murty have shown that infinitely many Fermat curves have infinite order Ceresa cycle \cite{EskandariMurty}, proving a weak version of Corollary \ref{cor: main infinitude} (where the genus is allowed to grow). If we fix $g$ but allow for curves defined over $\Qbar$ (and not necessarily over $\Q$), then the analogue of Corollary \ref{cor: main infinitude} holds for every genus $g \geq 3$ by \cite[Theorem 1.3.5]{ShouwuZhang}.

Prior to \cite{QiuZhang}, Beauville and Schoen showed that the Ceresa cycle of the curve $y^3z = x^4 + xz^3$ is torsion in the Griffiths group \cite{BeauvilleSchoen}. It was observed in \cite{LilienfeldtShnidman} that their proof applies to the bielliptic Picard curve $C_0$ as well. Theorem \ref{thm: torsion in griffiths intro} gives the first example of a {\it family} of plane quartic curves with this property. 

\subsection{Further questions}

Let $V_{g}$ be the subset of the coarse moduli space of curves $\mathcal{M}_g$ over $\Q$ consisting of those (geometric isomorphism classes of) curves $C$ such that $\kappa(C)\in \CH_1(J)\otimes \Q$ vanishes.
This is a countable union of closed algebraic subvarieties \cite[Lemma 5.1]{LagaShnidman-VanishingcriteriaCeresa}.
Can one describe the positive dimensional components of the torsion locus $V_g$? Is $V_g$ Zariski dense in $\mathcal{M}_g$?
After the current paper was posted, several papers have appeared that study this torsion locus and its variants; see \cite{GaoZhang-NorthcottCeresa, LagaShnidman-VanishingcriteriaCeresa, KerrTayou-CeresaHodgetheory}.

\subsection{Acknowledgements}
We thank Henri Darmon, David Lilienfeldt, and Padma Srinivasan for helpful conversations and remarks.  We thank Adam Logan for numerically testing Theorem \ref{thm:main} (while it was still a conjecture), using the forthcoming \cite{EllenbergLoganSrinivasan}. We thank Dick Gross, whose question helped motivate us to prove Theorem \ref{thm:main}.  We thank Tim Dokchitser and Drew Sutherland for helping us compute the central $L$-values mentioned in \S\ref{sec: Griffiths group}. Finally, we thank the referees for their careful reading of the paper.
This research was carried out while the first author was a Research Fellow at St John's College, University of Cambridge.
The second author was funded by the European Research Council (ERC, CurveArithmetic, 101078157).

\subsection{Notation and conventions}\label{subsec: notations and conventions}

\begin{itemize}
    \item All fields considered in this paper will be of characteristic $\neq 2, 3$.
    Given a field $k$, we denote by $k^{\mathrm{sep}}\subset \bar{k}$ a choice of separable and algebraic closure respectively, and the absolute Galois group by $\Gal_k = \Gal(k^{\mathrm{sep}}/k)$.
    \item A \define{variety} over a field $k$ is a separated scheme of finite type over $k$. A \define{curve} is a smooth, projective and geometrically integral variety of dimension $1$.

    \item Let $X$ be a smooth and quasi-projective scheme over a Dedekind domain (which might be a field). Then the Chow groups $\CH^p(X)$ of codimension $p$ cycles on $X$ (modulo rational equivalence and with $\Z$-coefficients) are defined and come equipped with flat pullbacks, proper pushforwards and an intersection product, see \cite[\S20.2]{Fulton-intersectiontheory}. If $Z\subset X$ is a codimension $p$ reduced closed subscheme write $[Z]$ for its class in $\CH^p(X)$.

    \item If $X$ is a smooth projective variety over a field $k$ we write $\CH^p(X)_{\mathrm{hom}} \subset \CH^p(X)$ for the classes in the kernel of the $\ell$-adic cycle class map for every prime $\ell$ invertible in $k$. 

    \item If $T\rightarrow S$ is a morphism of schemes and $X$ is an $S$-scheme we write $X_T$ for the $T$-scheme $X\times_S T$ and $X(T)$ for the set of sections of $X_T\rightarrow T$. If $T = \Spec(R)$ is affine we sometimes write $X_R$ and $X(R)$ instead of $X_T$ and $X(T)$. 
    \item If $A, B\rightarrow S$ are abelian schemes we denote by $\Hom(A, B)$ the set of $S$-homomorphisms between $A$ and $B$ and $\Hom^0(A,B) = \Hom(A,B) \otimes \Q$.
    We similarly write $\End(A) = \Hom(A, A)$ and $\End^0(A) = \End(A)\otimes \Q$.
    We write $A^{\vee}\rightarrow S$ for the dual abelian scheme (see \cite[Chapter 1, \S1]{faltingschai}) and write $\Hom^{\sym}(A, A^{\vee})$ for the subset of all self-dual homomorphisms $A\rightarrow A^{\vee}$.
    \item If $A/S$ is an abelian scheme we denote by $(n)$ or $(n)_A$ the multiplication-by-$n$ endomorphism $A\rightarrow A$.
    \item Group actions are always assumed to be \emph{left} actions.
\end{itemize}

\section{Ceresa and modified diagonal cycles}\label{sec: ceresa and GS cycles}

We recall the definitions of Ceresa and modified diagonal cycles, and show, by means of an example, a way to verify that $\kappa_e(C)$ is infinite order, for genus 3 curves over $\Qbar$. The latter is based on the methods of \cite{BuhlerSchoenTop} and \cite{EllenbergLoganSrinivasan}, which are in turn based on ideas of Bloch \cite{BlochCrelleI}. In this section, and for the rest of the paper, we work over a general (not necessarily algebraically closed) field $k$ of characteristic $\neq 2,3$. 

\subsection{Ceresa and modified diagonal cycles}\label{subsec: ceresa and GS cycles def}

Let $C/k$ be a curve of genus $g\geq 2$ with Jacobian $J$.
Let $e$ be a degree-$1$ divisor on $C$ and let $\iota_e\colon C\rightarrow J$ be the Abel--Jacobi map based at $e$.
We define the Ceresa cycle based at $e$ to be
\begin{align*}
    \kappa_e(C) := [\iota_e(C)] - (-1)_*[\iota_e(C)]\in \CH_1(J).
\end{align*}

The modified diagonal cycle of Gross and Schoen \cite{GrossSchoen} is a close relative of the Ceresa cycle. For $e \in C(k)$,  the \define{modified diagonal cycle} $\Delta_{GS,e} = \Delta_{GS,e}(C)$ based at $e$ is   
    \[\Delta_{GS,e} := \Delta_{123} - \Delta_{12} - \Delta_{13} - \Delta_{23} + \Delta_1 + \Delta_2 + \Delta_3 \in \CH^2(C^3),\]
where $\Delta_{123} = \{(c,c,c) \colon c \in C\}$, $\Delta_{12} = \{(c,c,e) \colon c \in C\}$, and so on. 
For a general degree $1$ divisor $e = \sum_i n_i [p_i]$ on $C$, we define $\Delta_{GS,e}(C) = \sum_{i} n_i \Delta_{GS,p_i}\in \CH^2(C^3)$.

We collect some useful facts about the Ceresa and modified diagonal cycles.

\begin{proposition}
The cycle classes $\kappa_e(C)$ and $\Delta_{GS,e}(C)$ are homologically trivial.    
\end{proposition}
\begin{proof}
    For $\kappa_e(C)$, this is immediate from the fact that $(-1)^*$ acts trivially on even dimensional cohomology. For $\Delta_{GS,e}(C)$, this is \cite[Proposition 3.1]{GrossSchoen}.
\end{proof}

\begin{proposition}\label{prop: Ceresa/GS torsion implies (2g-2)e canonical}
    Suppose that either $\kappa_e(C)\in \CH_1(J)$ or $\Delta_{GS,e}(C)\in \CH^2(C^3)$ is torsion. 
    Then $(2g-2)e - K_C$ is torsion in $\CH_0(C)$, where $K_C$ is the canonical divisor class.
\end{proposition}
\begin{proof}
    If $\kappa_e(C)$ is torsion, then this follows from \cite[Lemma 2.10]{LagaShnidman-VanishingcriteriaCeresa}, but we give an alternative proof here.
    If $D$ is a divisor class of degree $g-1$ on $C$, let $\Theta_D$ denote the image of the map $\Sym^{g-1}(C) \rightarrow J$ defined by $x\mapsto [x]-D$.
    If $\kappa_e(C)$ is torsion, then $[\iota_e(C)] = (-1)_*[\iota_e(C)]$ in $\CH_1(J)\otimes \Q$.
    Taking the $(g-1)$-fold Pontryagin product of this equality and dividing by $(g-1)!$, we get $[\Theta_{(g-1)e}] = (-1)_*[\Theta_{(g-1)e}]$ in $\CH_1(J)\otimes {\Q}$.
    By Riemann--Roch, $(-1)_*[\Theta_{(g-1)e}] = [\Theta_{K_C- (g-1)e}]$.
    Since $\Theta_{(g-1)e}$ defines a principal polarization, the map $\varphi\colon J(k)\rightarrow \CH^1(J)_{\hom}$ defined by $x\mapsto [\Theta_{(g-1)e + x}] - [\Theta_{(g-1
    )e}]$ is an isomorphism.
    Since 
    \[\varphi(K_C-(2g-2)e)=[\Theta_{K_C- (g-1)e}]-[\Theta_{(g-1)e}] = (-1)_*[\Theta_{(g-1)e}]-[\Theta_{(g-1)e}]\] 
    is torsion in $\CH^1(J)_{\hom}$, the divisor class $K_C-(2g-2)e$ is torsion.

    If $\Delta_{GS,e}(C)$ is torsion, this follows from \cite[Proposition 2.3.2(1)]{QiuZhang} (where it is assumed that $k=\C$, but whose proof is valid over any field $k$).
\end{proof}

\begin{lemma}\label{lemma: if e-e' torsion, then Ceresa cycles are same}
    If $e,e'$ are degree-$1$ divisors on $C$ such that $e-e'$ is torsion in $\CH_0(C)$, then $\kappa_e(C)-\kappa_{e'}(C)\in \CH_1(J)$ and $\Delta_{GS,e}(C)-\Delta_{GS,e'}(C)\in \CH^2(C^3)$ are both torsion.
\end{lemma}
\begin{proof}
    The fact that $\kappa_e(C)-\kappa_{e'}(C)$ is torsion follows from \cite[Lemma 2.11]{LagaShnidman-VanishingcriteriaCeresa}, and the fact that $\Delta_{GS,e}(C)-\Delta_{GS,e'}(C)$ is torsion follows from \cite[Proposition 3.6]{GrossSchoen}.
\end{proof}

If $k$ is algebraically closed, choose a degree-$1$ divisor $e$ such that $(2g-2)e = K_C$ in $\CH_0(C)\otimes {\Q}$ and define $\kappa(C)$ to be the image of $\kappa_e(C)$ in $\CH_1(J)\otimes {\Q}$, and $\Delta_{GS}(C)$ to be the image of $\Delta_{GS,e}(C)$ in $\CH^2(C^3)\otimes \Q$.
Lemma \ref{lemma: if e-e' torsion, then Ceresa cycles are same} shows that $\kappa(C)$ is independent of the choice of $e$, and Proposition \ref{prop: Ceresa/GS torsion implies (2g-2)e canonical} shows that $\kappa(C)= 0$ if and only if $\kappa_e(C)$ is torsion for some degree-$1$ divisor $e$. 
Similar remarks apply to $\Delta_{GS}(C)$.


\begin{theorem}[Shou-Wu Zhang]\label{thm: ceresa GS relation}
    For every degree-$1$ divisor $e$ on $C$, $\kappa_e(C)$ is torsion if and only if $\Delta_{GS,e}(C)$ is torsion.
\end{theorem}
\begin{proof}
    By Proposition \ref{prop: Ceresa/GS torsion implies (2g-2)e canonical}, we may assume $(2g-2)e = K_C$ in $\CH_0(C)\otimes \Q$.
    The result then follows from \cite[Theorem 1.5.5]{ShouwuZhang}.
\end{proof}

\subsection{A non-vanishing criterion}
Assume now that $k$ is a number field. Since $\Delta_{GS,e}$ is homologically trivial, we may consider for every prime $\ell$ its image under the $\ell$-adic Abel-Jacobi map 
\[\AJ_\ell \colon \CH^2(C^3)_{\mathrm{hom}} \longrightarrow \HH^1(\Gal_k, \HH^3(C_{k^{\sep}}^3, \Z_\ell(2)));\]
see \cite[\S9]{Jannsen} for more details concerning this map.
The following proposition is based on the methods of  \cite{BuhlerSchoenTop} and \cite{EllenbergLoganSrinivasan}. 
    \begin{proposition}\label{prop:torsion criterion}
        Let $e \in C(k)$ and let $v$ be a finite place of $k$ of good reduction for $C$. Let $\ell > 3$ be a prime coprime to the norm of $v$. Let $(C_v,e_v)$ be the smooth reduction of $(C,e)$ over the finite residue field $\F_v$. Define
        \begin{align}\label{equation: sum of points mod p jacobian}
        D := \sum_{c \in C_v(\F_v)} (c - e_v) \in J(\F_v).    
        \end{align}
        If the order of $D$ is a multiple of $\ell$, then $\AJ_\ell(\Delta_{GS,e}) \neq 0$.
    \end{proposition}    
\begin{proof} Consider the correspondence \[\Gamma = \{(a,\mathrm{Fr}(a), b,b) \colon a,b \in C_v\} \subset C^3_v \times C_v,\] where $\mathrm{Fr}$ is the Frobenius morphism.  This gives a homomorphism 
    \[\Gamma_* \colon \CH^2(C^3_v) \to \CH^1(C_v) \simeq \Pic(C_v),\]
    defined by $\Gamma_*\beta = p_{2*}([\Gamma] \cdot p_1^*\beta)$, where $p_i$ are the projections from $C^3_v \times C_v$. Let $\Delta_{GS,v}$ be the modified diagonal cycle on $C_v^3$ with respect to $e_v$. 
    A short computation yields $\Gamma_*(\Delta_{GS,v}) = \sum_{c \in C_v(\F_v)} (c - e_v)  = D \in J(\F_v)$.
    
    If $\AJ_{\ell}(\Delta_{GS,e})=0$, then a specialization argument using \cite[Proposition 3.1]{BuhlerSchoenTop} shows that $\AJ_{\ell}(D) \in \HH^1(\Gal_{k_v},\HH^1(C_{k_v^{\mathrm{sep}}}, \Z_{\ell}(1)))$ is zero too. 
    By \cite[Lemma 1.12]{BuhlerSchoenTop} this implies that $D$ is zero in $J(k_v)\otimes \Z_{\ell}$, so $D$ has order coprime to $\ell$.
    We conclude that if $D$ has order divisible by $\ell$ then $\AJ_{\ell}(\Delta_{GS,e})\neq 0$.
\end{proof}

\subsection{An infinite order example}
We now show by example how  Proposition \ref{prop:torsion criterion} can be used to verify that $\kappa_e(C)$ has infinite order.
For the rest of the section, let $k = \Q$, let $C$ be the curve with plane model $y^3 = x^4 + x^2 + 1$, and let $\ell = 7$. Consider the modified diagonal cycle $\Delta_{GS,\infty}$ based at the unique point $\infty$ at infinity; note that $(2g-2)\infty =  4\infty$ is canonical.  

\begin{lemma}\label{lem: order multiple of 7}
    $\AJ_7(\Delta_{GS,\infty}) \neq 0$.
\end{lemma}
\begin{proof}
    We take $v = 41$ in Proposition \ref{prop:torsion criterion} and show that $D\in J(\F_{41})$ of \eqref{equation: sum of points mod p jacobian} has order divisible by $7$ by explicit computation.  Let $\pi \colon C \to E \colon y^3 = x^2 + x + 1, (x,y)\mapsto (x^2,y)$ be the double cover and let $\tau \in \Aut(C)$ be the corresponding covering involution. Pairing points $c$ and $\tau(c)$ together, we see that $D$ is a sum of points on $\pi^*\Pic^0(E)$ plus some leftover ramification points. Thus $2D$ lies in $\pi^*\Pic^0(E)$, and identifying $\Pic^0(E)$ with $E$, we have
    \[2D = \pi^*\left((0,1) + 2\sum_{q \in E(\F_{41})_{\mathrm{lift}}} q\right),\]
    where $E(\F_{41})_{\mathrm{lift}}$ is the set of non-branch points in the image of $\pi \colon C(\F_{41}) \to E(\F_{41})$. We compute in Sage that the interior sum equals $(37,15) \in E(\F_{41})$ which has order $7$.
    Therefore $2D = \pi^*((37,15))$ has order $7$ too, hence the order of $D$ is divisible by $7$.
    \end{proof}

Let $\Sym^3(C) := C^3/S_3$ be the symmetric cube of $C$, and let $V := \HH^3(\Sym^3(C_{\overline{\Q}}), \Z_7(2))$. 

\begin{lemma}\label{lem: explicit V}
$V \simeq \HH^3(J_{\overline{\Q}}, \Z_7(2)) \oplus \HH^1(C_{\overline{\Q}}, \Z_7(1))$ as $\Gal_{\Q}$-representations. 
\end{lemma}

\begin{proof}
 Since $C$ is a nonhyperelliptic genus $3$ curve, the Abel--Jacobi map $\Sigma \colon \Sym^3(C) \to J$ sending a degree $3$ divisor $D$ to $D - 3\infty$ is the blow-up of $J$ along the image of those $D$ with the property that $\dim \HH^0(C, \calO(D))=2$ \cite[Theorem 2.3]{MunozPorras-abeljacobi}. By Riemann-Roch, these are exactly the degree-$3$ divisors linearly equivalent to $4\infty-x$ for some unique $x\in C$. 
  We conclude that $\Sigma$ is the blow-up of $J$ along the closed subvariety $\{ \infty -x \mid x\in C\} \simeq C$. 
  The lemma now follows from the description of the cohomology of a blow-up \cite[VII, Th\'eor\`eme 8.1.1]{SGA5}. 
\end{proof}

\begin{lemma}\label{lem:section is not 0}
    The class $\Delta_{GS,\infty}$ has infinite order.
\end{lemma}
\begin{proof}
    Suppose, for the sake of contradiction, that $\Delta_{GS,\infty}$ has finite order.
    Then its Abel-Jacobi image $\AJ_7(\Delta_{GS,\infty}) \in \HH^1(\Gal_{\Q}, \HH^3(C^3_{\bar \Q}, \Z_7))$ has finite order. 
    By Lemma \ref{lem: order multiple of 7} it is nonzero.  Let $h \colon C^3 \to \Sym^3(C)$ be the quotient map and let $\tilde{\Delta}_{GS} = h_*\Delta_{GS,\infty}$.
    Use the same notation as before for the Abel-Jacobi map $\AJ_\ell \colon \CH^2(\Sym^3(C))_{\mathrm{hom}} \to \HH^1(\Gal_\Q, V).$
    Since $h^*\tilde{\Delta}_{GS} = 6\Delta_{GS,\infty}$, we conclude that $\AJ_\ell(\tilde{\Delta}_{GS})$ is also nonzero and torsion.  On the other hand, we may use Magma to compute the characteristic polynomial of a Frobenius element acting on $\HH^1(C_{\Qbar}, \Z_7(1))$, and hence on $V$ as well (via Lemma \ref{lem: explicit V}). We check that the action of the geometric Frobenius $\mathrm{Fr}_{11}$ at $11$ on $V$
    satisfies
    \[\det(\mathrm{Fr}_{11} - 1) =29049104246323668435011663307177984 \not \equiv 0 \pmod{7}.\] 
    Thus $\HH^1(\Gal_\Q, V)$ is torsion-free by \cite[6.1]{Zelinsky}, which is a contradiction.
\end{proof}

    Combining Lemma \ref{lem:section is not 0} and Theorem \ref{thm: ceresa GS relation}, we have proven:

\begin{corollary}
    The Ceresa cycle $\kappa_{\infty}(C)$ of the curve $C\colon y^3 = x^4+x^2+1$ has infinite order.
\end{corollary}

\section{Bielliptic Picard curves}\label{sec: bielliptic Picard}

We recall some results on bielliptic Picard curves from \cite{LagaShnidman}; we refer to that paper for more details.
In particular, we collect the results we will need on the endomorphism algebra of the Prym variety and its monodromy in \S\ref{subsec: galois action NS group}.

\subsection{Definitions}
A \define{bielliptic Picard curve} $C$ over a field $k$ (of characteristic $\neq 2,3$) is a smooth projective plane quartic curve over $k$ with affine model of the form
\begin{align}\label{equation: bielliptic picard body of text}
    y^3 = x^4+ax^2+b,
\end{align}
for some $a,b \in k$. We often write $C = C_{a,b}$. Smoothness of $C$ implies that 
\begin{align*}
 \Delta_{a,b} := 16b(a^2 - 4b)\neq 0,
\end{align*}
and any plane quartic of the form \eqref{equation: bielliptic picard body of text} with $\Delta_{a,b}\neq 0$ is smooth. 
These curves admit a $\mu_6$-action given by $\zeta\cdot (x,y) = (\zeta^3x,\zeta^4y)$ for every $\zeta\in \mu_6$. 
The unique fixed point of this action is the unique point $\infty$ at infinity. 
Since $\infty$ intersects the line at infinity with multiplicity $4$, the divisor $4\infty$ is canonical.

Define the \define{$j$-invariant} of $C_{a,b}$ to be $j(C_{a,b}) := (4b - a^2)/4b \in k^\times$.
\begin{lemma}\label{lemma: bielliptic isomorphisms}\cite[\S2.2-3]{LagaShnidman}
    Let $C_{a,b}$ and $C_{a',b'}$ be bielliptic Picard curves over $k$.  
    \begin{enumerate}
        \item $C_{a,b} \simeq C_{a',b'}$ if and only if there exists $\lambda \in k^\times$ such $(a',b') = (\lambda^6a, \lambda^{12}b)$. 
        \item $C_{a,b}$ is isomorphic to $C_{a',b'}$ over $\overline{k}$ if and only if $j(C_{a,b}) = j(C_{a',b'})$. 
    \end{enumerate}
\end{lemma}

\begin{lemma}\label{lemma: j-invariant bielliptic picard curve field of definition}
    Let $C_{a,b}$ be a bielliptic Picard curve over $k^{\mathrm{sep}}$. Then $C$ is defined over $k$ (in other words, there exists a curve $X$ over $k$ such that $X_{k^{\mathrm{sep}}} \simeq C_{a,b}$) if and only if $j(C_{a,b}) \in k$.
\end{lemma}
\begin{proof}
    Suppose that $j = j(C_{a,b}) \in k$.
    If $j=1$, then $a=0$ and $C_{a,b}$ is $k^{\mathrm{sep}}$-isomorphic to $C_{0,1}$ and defined over $k$.
    If $j\neq 1$, then a computation using Lemma \ref{lemma: bielliptic isomorphisms}(1) shows that $C_{a,b}$ is $k^{\mathrm{sep}}$-isomorphic to $C_{2,(1-j)^{-1}}$ and hence also defined over $k$.
    Conversely, suppose that $C_{a,b}$ is defined over $k$.
    Then $C_{a,b}\simeq \sigma^*(C_{a,b}) =  C_{\sigma(a),\sigma(b)}$ for all $\sigma \in \Gal_k$.
    By Lemma \ref{lemma: bielliptic isomorphisms}(2), $j(C_{a,b}) = \sigma(j(C_{a,b}))$ for all $\sigma\in \Gal_k$. In other words, $j(C_{a,b}) \in k$.
\end{proof}

\subsection{The Prym variety}\label{subsec: prym varieties}
The curve $C$ admits an involution $\tau(x,y)=(-x,y)$, and the quotient by this involution is a double cover $\pi \colon C \to E$ to an elliptic curve $E$ with origin the image of $\infty$.
This elliptic curve has equation $y^3 = x^2+ax+b$ and (after a change of variables) short Weierstrass model $E \colon y^2 = x^3 + 16(a^2 - 4b)$. 
This double cover decomposes the Jacobian variety $J$ of $C$.
Let
\begin{align}\label{equation: definition prym variety}
    P := \ker(1+\tau^* \colon J\rightarrow J)
\end{align}
be the \define{Prym variety} associated to the double cover $C\rightarrow E$.
\begin{lemma}\cite[\S2.5]{LagaShnidman}\label{lemma: J isogenous to P times E}
    $P$ is an abelian surface and $J$ is isogenous to $P\times E$.
\end{lemma}

\subsection{Galois action on the N\'eron--Severi group}\label{subsec: galois action NS group}
The main ingredient needed from \cite{LagaShnidman} for our proof of Theorem \ref{thm:main} is an explicit description of the $\Gal_k$-action on the N\'eron--Severi group $\NS(P_{k^{\sep}})$, as well as the fact that $P_{k^{\sep}}$ has quaternionic multiplication. 

We will need a universal version of this story over the parameter space of all bielliptic Picard curves, so let
\[S := \A^2_{\Z[1/6]}\setminus \{\Delta_{a,b} = 0\} = \Spec(\Z[1/6,a,b,\Delta_{a,b}^{-1}]).\] 
Let $\mathcal{C} \rightarrow S$ be the universal bielliptic Picard curve with equation \eqref{equation: bielliptic picard body of text}.
The morphism $\mathcal{C}\rightarrow S$ is smooth, proper and of relative dimension $1$ and every bielliptic Picard curve occurs as a fiber of this morphism.
Let $\mathcal{J} = \mathbf{Pic}_{\mathcal{C}/S}^0$ be its relative Jacobian \cite[\S9.4, Proposition 4]{BLR-neronmodels}, an abelian scheme over $S$.
The definition \eqref{equation: definition prym variety} of the Prym variety works in families and we obtain an abelian scheme $\mathcal{P}\rightarrow S$ of relative dimension $2$ whose fibers are the Prym varieties described in \S\ref{subsec: prym varieties}.

We need to pass to a cover of $S$ to "see" all endomorphisms of $\mathcal{P}$.
Let $\tilde{S}\rightarrow S$ be the finite \'etale cover given by adjoining a primitive third root of unity $\omega$ and a sixth root of $\Delta_{a,b}$.
More formally, consider the ring extension 
\[R = \Z[1/6,a,b,\Delta_{a,b}^{-1}] \rightarrow \tilde{R} = \Z[\omega][1/6,a,b, \Delta_{a,b}^{-1}, \varepsilon]/(\varepsilon^6  -\Delta_{a,b})\] 
and define $\tilde{S}\rightarrow S$ to be the induced morphism of schemes.
This morphism is finite, \'etale, and Galois, with Galois group $G := \Aut(\tilde{S}/S)$ a dihedral group of order $12$.
The group $G$ naturally acts on $\End^0(\mathcal{P}_{\tilde{S}})$.
Let $\mathrm{triv}\colon G\rightarrow \GL_1(\Q)$ be the trivial representation and let $\mathrm{std}\colon G\rightarrow \GL_2(\Q)$ be a model for the reflection representation, i.e.\ the unique irreducible $2$-dimensional representation of $G$ defined over $\Q$.
See \S\ref{subsec: notations and conventions} for our notations concerning abelian schemes. 
\begin{theorem}\label{thm: Neron Severi action}
    The endomorphism algebra $B:=\End^0(\mathcal{P}_{\tilde{S}})$ is a quaternion algebra of discriminant $6$ over $\Q$. 
    Moreover, the $G$-representation $\Hom^{\sym}(\mathcal{P}_{\tilde{S}}, \mathcal{P}^{\vee}_{\tilde{S}}) \otimes \Q$ is three-dimensional and isomorphic to $\mathrm{triv} \oplus \mathrm{std}$. 
\end{theorem}
\begin{proof}
    The schemes $S$ and $\tilde{S}$ are integral and normal; let $K$ and $\widetilde{K}$ denote their function fields and let $P = \mathcal{P}_K$ denote the generic fiber of $\mathcal{P}$.
    By \cite[\S2, Lemma 1]{Faltings-finitenesstheorems}, the natural map $\End(\mathcal{P}_{\tilde{S}}) \rightarrow \End(P_{\widetilde{K}})$ is an isomorphism. By \cite[Tag \href{https://stacks.math.columbia.edu/tag/0BQM}{0BQM}]{stacksproject}, the natural map $G = \Aut(\tilde{S}/S) \rightarrow \Gal(\widetilde{K}/K)$ is also an isomorphism.
    It therefore suffices to prove all claims over the generic point $\Spec(K)$. 
    By \cite[Lemma 6.16]{LagaShnidman}, $P$ is geometrically simple.
    If $\bar{K}$ denotes an algebraic closure of $K$ then \cite[Lemma 6.2]{LagaShnidman} implies that $\End^0(P_{\bar{K}})$ is a discriminant $6$ quaternion algebra, and \cite[Corollary 6.6]{LagaShnidman} shows that $\End^0(P_{\bar{K}}) = \End^0(P_{\tilde{K}})$.
    The claim concerning $\Hom^{\sym}(\mathcal{P}_{\tilde{S}}, \mathcal{P}^{\vee}_{\tilde{S}}) \otimes \Q = \Hom^{\sym}(P_{\tilde{K}}, P^{\vee}_{\tilde{K}})\otimes \Q$ follows from \cite[Corollary 6.8]{LagaShnidman} and the fact that the latter is isomorphic to $\NS(P_{k^{\mathrm{sep}}})\otimes \Q$.
\end{proof}
The $\mu_3$-action on $\mathcal{J}$ restricts to a $\mu_3$-action on $\mathcal{P}$, which by functoriality induces a $\mu_3$-action on $\Hom^{\sym}(\mathcal{P}_{\tilde{S}}, \mathcal{P}_{\tilde{S}}^{\vee})$ via $\omega\cdot \lambda =  \omega^{\vee} \circ\lambda \circ \omega^{-1}$ for all $\lambda \in \Hom^{\sym}(\mathcal{P}_{\tilde{S}}, \mathcal{P}_{\tilde{S}}^{\vee})\otimes \Q$ and $\omega \in \mu_3$.
This $\mu_3$-action commutes with the $G$-action, so preserves the isotypic components $\mathrm{triv}$ and $\mathrm{std}$.
\begin{lemma}\label{lemma: description mu3-action on NS(P)}
    In the decomposition $\Hom^{\sym}(\mathcal{P}_{\tilde{S}}, \mathcal{P}^{\vee}_{\tilde{S}}) \otimes \Q = \mathrm{triv} \oplus \mathrm{std}$ of Theorem \ref{thm: Neron Severi action}, $\mu_3$ acts trivially on $\mathrm{triv}$ and $\mu_3$ acts on $\mathrm{std}\otimes \Qbar$ as a direct sum of the two distinct nontrivial characters.
\end{lemma}
\begin{proof}
    The proof of \cite[Corollary 6.8]{LagaShnidman} shows that if we write $B = \text{span}\{1,i,j,ij\}$ with $i^2=-3$, $j^2=2$, $ij=-ji$ and $\omega = (-1+i)/2$, then $\Hom^{\sym}(\mathcal{P}_{\tilde{S}}, \mathcal{P}^{\vee}_{\tilde{S}}) \otimes \Q \simeq \text{span}\{1,j,ij\}$, $\mathrm{triv}$ is spanned by $1$, $\mathrm{std}$ is spanned by $\{i,ij\}$ and $\omega \in \mu_3$ acts via conjugation by $\omega\in B$.
    The description of the $\mu_3$-action on $\mathrm{triv}$ and $\mathrm{std}$ now follow from an explicit calculation.
\end{proof}

\section{Background on Chow motives}\label{section: background chow motives}

Our proof of Theorem \ref{thm:main} uses the language of relative Chow motives, so we recall the necessary definitions and prove some basic statements on motives for which we could not find a reference.
The most important sections are \S\ref{subsec: relative chow motives} and \S\ref{subsection: motives of abelian schemes} and the others can be referred back to when needed in \S\ref{sec: motivic computations and proof of main theorem}.

We call a scheme $S$ \define{adequate} if it is smooth and quasiprojective over a Dedekind domain.
For the remainder of this section fix such a scheme $S$.

\subsection{Correspondences}
We denote by $\mathsf{SmProj}(S)$ the category of smooth projective $S$-schemes.
For $X$ in $\mathsf{SmProj}(S)$ the Chow groups $\CH^p(X)$ are well defined, see \S\ref{subsec: notations and conventions}. For the remainder of the paper write $\CH^p(X; \Q) := \CH^p(X)\otimes_{\Z} \Q$ and if $X\rightarrow S$ has equidimensional fibers write $d(X/S)$ for the relative dimension of $X\rightarrow S$.
If $X$ and $Y$ are in $\mathsf{SmProj}(S)$ we define the $\Q$-vector space of \define{correspondences} by $\Corr(X,Y) := \CH^{*}(X\times_S Y;\Q)$.
We say a correspondence has \define{degree $r$} if it lies in the subspace
\[
\Corr^r(X,Y) := \bigoplus_{i} \CH^{r+d(X_i/S)}(X_i\times_S Y;\Q)
\]
where $X = \sqcup_i X_i$ is a decomposition such that $X_i\rightarrow S$ is equidimensional.
If $p\in \Corr^r(X, Y)$ and $q\in \Corr^s(Y,Z)$ are correspondences, denote by $q\circ p = \pi_{XZ,*}(\pi_{XY}^*p \cdot \pi_{YZ}^* q) \in \Corr^{r+s}(X,Z)$ their composition.
For an $S$-morphism $f\colon X\rightarrow Y$, let $\Gamma_f \subset X\times_S Y$ be its graph and let $^t\Gamma_f\subset Y\times_S X$ be the transpose.
Write $\Delta_{X/S}$ for the graph of the identity $X\rightarrow X$.

\subsection{Relative Chow motives}\label{subsec: relative chow motives}

We denote by $\mathsf{Mot}(S)$ the category of relative Chow motives with respect to graded correspondences over $S$, see \cite[\S1]{DeningerMurre} for more details.
The results of \cite{DeningerMurre} are stated under the standing assumption that $S$ is smooth and quasi-projective over a field, but their results continue to hold in our more general set-up since intersection theory is developed in this generality, see \cite[\S1, Remark 1.1]{Kunnemann-arakelovchowgroups}.

We recall that an object of $\mathsf{Mot}(S)$ is a triple $(X/S, p, m)$, where $X/S$ is a smooth projective $S$-scheme, $p\in \Corr^{0}(X,X)$ is an idempotent correspondence and $m$ is an integer. 
Morphisms are given by
\[
\Hom((X/S,p,m),(Y/S,q,n)) := q\circ \Corr^{m-n}(X, Y)\circ p.
\]
Given a motive $M = (X/S,p,m)$ and $n\in \Z$ we write $M(n) := (X/S,p,m+n)$.
There is a contravariant functor $\mathsf{SmProj}(S) \rightarrow \mathsf{Mot}(S)$ sending $X/S$ to $\frak{h}(X/S) := (X/S, [\Delta_{X/S}], 0)$ and sending morphisms $f\colon X\rightarrow Y$ to $[^t\Gamma_f]$. 
Write $\mathbf{1}= \mathbf{1}_S := \frak{h}(S/S)$ for the unit motive and $\mathbb{L} := \mathbf{1}(-1)$ for the Lefschetz motive.
There is a notion of direct sums and tensor products in $\mathsf{Mot}(S)$.

\subsection{Base changing motives}\label{subsec: base changing motives}
Suppose that $T$ is another adequate scheme and $f\colon T\rightarrow S$ a morphism.
There are two situations in which we have a good theory of pullbacks\footnote{Pullbacks can be defined in greater generality, see for example \cite[Example 20.1.2]{Fulton-intersectiontheory}, but we restrict ourselves to the concrete cases we need in this paper.} \cite[\S20.1]{Fulton-intersectiontheory}:
\begin{enumerate}
    \item If $f$ is flat, there are pullback maps $f^*\colon \CH^p(X) \rightarrow \CH^p(X_T)$ for every $X/S$ in $\mathsf{SmProj}(S)$ and $p\in \Z_{\geq 0}$.
    \item If $f$ is a regular embedding, there are Gysin homomorphisms $f^{*}\colon \CH^p(X) \rightarrow \CH^p(X_T)$ for every $X/S$ in $\mathsf{SmProj}(S)$ compatible with flat pullbacks.
\end{enumerate}
We will say that $f$ \define{admits pullbacks} if it is a finite composition $f_1\circ \cdots \circ f_n$ of morphisms of the above form, and for such a morphism we define $f^* = f_n^*\circ \cdots \circ f_1^*\colon \CH^p(X)\rightarrow \CH^p(X_T)$ for every $X/S$ in $\mathsf{SmProj}(S)$. (This does not depend on the choice of $f_i$'s.)
For example, if $T$ is the spectrum of a field then $f$ admits pullbacks.
When $f$ admits pullbacks we define a functor 
\begin{align*}
    f^*\colon \mathsf{Mot}(S)\rightarrow \mathsf{Mot}(T),
\end{align*}
which on objects sends $(X/S,p,m)$ to $(X_T/T, f^*(p), m)$ and sends a correspondence $q$ to $f^*(q)$.

\subsection{Descending direct summands of motives}\label{subsec: descending direct summands of motives}
Let $\pi\colon \tilde{S}\rightarrow S$ be a finite \'etale morphism of adequate schemes.
Suppose additionally that it is Galois and write $G$ for the group of Deck transformations, so we have an isomorphism of $S$-schemes $g\colon \tilde{S} \rightarrow \tilde{S}$ for every $g\in G$.
Using the pullback functor from \S\ref{subsec: base changing motives}, we obtain functors $g^* \colon \mathsf{Mot}(\tilde{S}) \rightarrow \mathsf{Mot}(\tilde{S})$. 
Let $M,N$ be objects of $\mathsf{Mot}(S)$ and write $\tilde{M} = \pi^*M$, $\tilde{N} = \pi^*N$ for the corresponding objects of $\mathsf{Mot}(\tilde{S})$.
Then we have canonical identifications $\tilde{M} = g^*\tilde{M}$ and $\tilde{N} = g^*\tilde{N}$ for all $g\in G$.
It follows by functoriality that $\Hom(\tilde{M},\tilde{N})$ inherits a $G$-action.

\begin{lemma}\label{lemma: descending morphisms of motives}
    In the above notation, the natural map $\Hom(M,N) \rightarrow \Hom(\tilde{M},\tilde{N})^G$ is an isomorphism. 
\end{lemma}
\begin{proof}
    This is well known when $S = \Spec(k)$ \cite[Lemma 1.17]{Scholl-classicalmotives}; we sketch the generalization of the classical proof to the relative setting. 
    Write $M = (X/S, p, m)$ and $N = (Y/S,q,n)$. 
    Then an element of $\Hom(\tilde{M},\tilde{N})^G$ is the same as a $G$-invariant element $\alpha$ of $q\circ \CH^{d(X/S)+m-n}(X\times_S Y;\Q)\circ p$.
    Choose a cycle $Z$ on $X_{\tilde{S}} \times_{\tilde{S}} Y_{\tilde{S}}$ representing $\alpha$.
    Then $\sum_{g\in G} g^*Z$ descends to a unique cycle $W$ of $X\times_S Y$. 
    Moreover the rational equivalence class $[W]$ is independent of the choice of $Z$ representing $\alpha$.
    It follows that $\alpha \mapsto \frac{1}{|G|} [W]$ is a well-defined inverse to the base change map above.
\end{proof}

This implies that Galois-invariant summands also satisfy descent:

\begin{corollary}\label{corollary: descending direct summands of motives}
    Let $M$ be an object of $\mathsf{Mot}(S)$ and $\tilde{U}\oplus \tilde{V}$ a direct sum decomposition of $\tilde{M} = \pi^*M$ such that we have an equality of direct summands $g^*\tilde{U} = \tilde{U}$ and $g^*\tilde{V} = \tilde{V}$ in $g^*\tilde{M} = \tilde{M}$ for all $g\in G$.
    Then there exists a unique direct sum decomposition $M = U\oplus V$ such that $\pi^*U = \tilde{U}$ and $\pi^*V = \tilde{V}$.
\end{corollary}
\begin{proof}
    Since direct sum decompositions correspond bijectively to a pair of orthogonal idempotents in the endomorphism ring, this follows immediately from Lemma \ref{lemma: descending morphisms of motives}.
\end{proof}

\subsection{Group actions on motives}\label{subsec: group actions on motives}

Let $M$ be a motive in $\mathsf{Mot}(S)$ and $G$ a finite group acting on $M$.
Then every idempotent in the group algebra $\Q[G]$ gives rise to an idempotent in $M$.
For example, if $\rho$ is an irreducible (but not necessarily absolutely irreducible) representation defined over $\Q$ and $e_{\rho}$ is the corresponding central idempotent, we get a direct summand $M_{\rho}$ of $M$, called the \define{$\rho$-isotypic component}.
Applying this to $\rho = $ the trivial representation, we have $e_{\rho} = \frac{1}{|G|}\sum_{g\in G} g$ and we see that the \define{motive of $G$-fixed points} $M^G = M_{\rho}$ is well-defined and a direct summand of $M$.

We may also consider actions by nonconstant finite group schemes, using Lemma \ref{lemma: descending morphisms of motives}.
For example, if $n$ is invertible in $S$, a $\mu_n$-action on a motive $M$ in $\mathsf{Mot}(S)$ is an $\Aut(\tilde{S}/S)$-equivariant map $\mu_n(\tilde{S})\rightarrow \Aut(f^*M)$, where $f\colon \tilde{S} = S\times_{\Z[1/n]} \Z[\zeta_n][1/n] \rightarrow S$.
The $\Aut(\tilde{S}/S)$-equivariance implies that the $\mu_n$-fixed points of $f^*M$ descend to a well-defined direct summand $M^{\mu_n}$ of $M$.

\subsection{Chow groups of motives}\label{subsec: chow groups of motives}
The (covariant) \define{Chow groups} of a motive $M$ in $\mathsf{Mot}(S)$ are by definition the $\Q$-vector spaces $\CH^p(M) := \Hom(\mathbb{L}^{\otimes p}, M)$ for $p\in \Z$, where we recall from \S\ref{subsec: relative chow motives} that $\mathbb{L}$ denotes the Lefchetz motive.
For example, $\CH^p(\frak{h}(X/S)) = \CH^p(X;\Q)$ for all $X/S$ in $\mathsf{SmProj}(S)$.

Suppose a finite group $G$ acts on $M$.
This induces an action of $G$ on $\CH^p(M)$ for every $p$.
If $e$ is an idempotent in $\Q[G]$ corresponding to direct summand $e\cdot M$ of $M$, then $\CH^p(e\cdot M) = e\cdot \CH^p(M)$.
In particular, $\CH^p(M^G) = \CH^p(M)^G$.

\subsection{Relative Artin motives}
Assume that $S$ is connected.
An \define{Artin motive} in $\mathsf{Mot}(S)$ is a direct summand of a motive of the form $\frak{h}(X/S)$, where $X\rightarrow S$ in $\mathsf{SmProj}(S)$ has relative dimension zero or equivalently, is finite \'etale.
When $S$ is the spectrum of a field, Artin motives are the same as Galois representations on finite-dimensional $\Q$-vector spaces; we give a similar (presumably well known) description in the relative setting.
Let $f\colon \tilde{S} \rightarrow S$ be a connected finite \'etale cover with Galois group $G = \Aut(\tilde{S}/S)$.
We say an Artin motive $M$ in $\mathsf{Mot}(S)$ is \define{trivialized by $f$} if $f^*M \simeq \mathbf{1}_{\tilde{S}}^{\oplus n}$ for some $n\in \Z_{\geq 0}$.
In that case the $\Q$-vector space $\Hom(\mathbf{1}_{\tilde{S}},f^*M)$ is $n$-dimensional (since $\tilde{S}$ is connected) and has a natural $G$-action, by the same reasoning as \S\ref{subsec: descending direct summands of motives}.
\begin{lemma}\label{lemma: artin motives are local systems with finite monodromy}
    Every Artin motive in $\mathsf{Mot}(S)$ is trivialized by some connected finite \'etale Galois cover $f\colon \tilde{S}\rightarrow S$.
    Fixing such an $f$, the assignment $M\mapsto V_M:=\Hom(\mathbf{1}_{\tilde{S}},f^*M)$ induces an equivalence between the full subcategory of Artin motives in $\mathsf{Mot}(S)$ trivialized by $f$ and the category of $\Aut(\tilde{S}/S)$-representations on finite-dimensional $\Q$-vector spaces.
\end{lemma}
\begin{proof}(Sketch) 
    If $M$ is a direct summand of $\frak{h}(X/S)$, then $M$ is trivialized by the compositum of the Galois closures of the connected components of $X$, justifying the first sentence.
    Fully faithfulness of $M\mapsto V_M$ follows from Lemma \ref{lemma: descending morphisms of motives} and the fact that $\Hom(f^*M,f^*N)\rightarrow \Hom(V_M,V_N)$ is an isomorphism when $\tilde{S}$ is connected. 
    Moreover, again since $\tilde{S}$ is connected $V_{\frak{h}(\tilde{S}/S)}$ is isomorphic to the regular representation of $G$.
    Therefore essential surjectivity follows from the fact that every $G$-representation is a direct summand of copies of the regular representation.
\end{proof}

\subsection{Motives of abelian schemes}\label{subsection: motives of abelian schemes}

Let $A\rightarrow S$ be an abelian scheme of relative dimension $g$.
Deninger--Murre \cite[\S3]{DeningerMurre} have shown that the motive $\frak{h}(A/S)$ has a canonical decomposition into Chow--Künneth components.

\begin{theorem}[Deninger--Murre]\label{theorem: deninger--Murre projectors}
    There exists a unique direct sum decomposition
    \begin{align}\label{equation: deninger-murre chow-kunneth}
        \frak{h}(A/S) = \bigoplus_{i=0}^{2g} \frak{h}^i(A/S)
    \end{align}
    in $\mathsf{Mot}(S)$ with the property that $[^t\Gamma_{(n)}]$ acts as $n^{i}$ on $\frak{h}^i(A/S)$ for every $0\leq i\leq 2g$ and every $n\in \Z$.
\end{theorem}
We call the decomposition \eqref{equation: deninger-murre chow-kunneth} the \define{canonical Chow--Künneth decomposition} of $\frak{h}(A/S)$.
In what follows $\frak{h}^i(A/S)$ will always denote the $i$-th component of this decomposition.

We will use the next proposition frequently when doing explicit calculations with $\frak{h}^1$ of abelian schemes.

\begin{proposition}\label{proposition: AVs up to isogeny same as h^1} \cite[Proposition 2.2.1]{Kings-higherregulators}
    Let $A,B\rightarrow S$ be abelian schemes of the same relative dimension. 
    Then the pullback map 
    \begin{align}\label{equation: AV to h^1 is fully faithful}
        \Hom^0(A, B) \rightarrow \Hom(\frak{h}^1(B/S), \frak{h}^1(A/S))
    \end{align}
    is an isomorphism of $\Q$-vector spaces.
\end{proposition}

We also include the following result on abelian schemes, which follows from \cite[\S2, Lemma 1]{Faltings-finitenesstheorems} and the fact that adequate schemes are normal.
\begin{proposition}\label{prop: abelian schemes normal bases}
    Suppose that $S$ is integral with generic point $\Spec(K)$.
    Then the map $A/S\mapsto A_K/K$ identifies the category of abelian schemes over $S$ with a full subcategory of the category of abelian varieties over $K$.
\end{proposition}

\section{Proof of main theorem}\label{sec: motivic computations and proof of main theorem}

This section is the technical heart of the paper. 
In \S\ref{subsec: an isomorphism of motives} we state and prove the isomorphism of motives \eqref{equation: intro iso of chow motives} promised in the introduction.
Next in \S\ref{subsec: chow groups} we translate this isomorphism into a concrete statement about Chow groups, and in \S\ref{subsec: image of ceresa cycle} we show that one of these Chow groups contains the Ceresa cycle.
In \S\ref{subsec: identifying sigma} we identify the image of the Ceresa cycle under our constructed isomorphism.
In the last two subsections we combine everything and prove the theorems and corollaries stated in the introduction.

\subsection{An isomorphism of motives}\label{subsec: an isomorphism of motives}

Let $C=C_{a,b}$ be a bielliptic Picard curve \eqref{equation: bielliptic picard body of text} over a field $k$ with Jacobian $J$.
The $\mu_6$-action on $C$ induces, via pullback, a $\mu_6$-action on $J$ and $\frak{h}(J)$.
The uniqueness of the Chow--Künneth decomposition \eqref{equation: deninger-murre chow-kunneth} shows that this $\mu_6$-action restricts to a $\mu_6$-action on $\frak{h}^i(J)$ for each $0\leq i\leq 6$.
Let $E^{\Delta} \colon y^2 = x^3 + 4b(a^2 - 4b)^2$ be the sextic twist of $E:y^2 = x^3+16(a^2-4b)$ by $\Delta = \Delta_{a,b} = 16b(a^2 - 4b)$.
Recall from \S\ref{subsec: group actions on motives} that $\frak{h}^3(J)^{\mu_6}$ is a direct summand of $\frak{h}^3(J)$.
\begin{theorem}\label{thm: isomorphism of motives over a field}
    There is an isomorphism in $\mathsf{Mot}(k)$: $\frak{h}^3(J)^{\mu_6} \simeq \frak{h}^1(E^{\Delta})(-1)$.
\end{theorem}
In fact, it will be crucial for our proof of Theorem \ref{thm:main} to construct this isomorphism universally. 
We will use the notation of \S\ref{subsec: galois action NS group}, so $S = \A^2_{\Z[1/6]}\setminus \{\Delta_{a,b} = 0\}$ is the base scheme and $\mathcal{C} \rightarrow S$ the universal bielliptic Picard curve with Jacobian $\mathcal{J}\rightarrow S$ and Prym variety $\mathcal{P}\rightarrow S$.
Let $\mathcal{E}\rightarrow S$ be the relative elliptic curve with equation $y^2 = x^3 +16(a^2-4b)$ and let $\mathcal{E}^\Delta\rightarrow S$ be the relative elliptic curve with equation $y^2 = x^3 + 4b(a^2 - 4b)^2$.
There is again a $\mu_6$-action on $\mathcal{C}$, $\mathcal{J}$ and $\frak{h}^3(\mathcal{J}/S)$. 
The goal of this subsection is to show:
\begin{theorem}\label{theorem: iso relative chow motives}
    There is an isomorphism in $\mathsf{Mot}(S)$: $\frak{h}^3(\mathcal{J}/S)^{\mu_6}\simeq \frak{h}^1(\mathcal{E}^{\Delta}/S)(-1)$.
\end{theorem}

Specializing this isomorphism at  $k$-points $\Spec(k) \rightarrow S$ recovers Theorem \ref{thm: isomorphism of motives over a field}.
The proof of Theorem \ref{theorem: iso relative chow motives} will be given at the end of this subsection after combining a few different ingredients.

The starting point is the decomposition
\[
\frak{h}^1(\mathcal{J}/S) = \frak{h}^1(\mathcal{P}/S)  \oplus  \frak{h}^1(\mathcal{E}/S)
\]
induced by the isogeny $\mathcal{J}\sim \mathcal{P} \times_S \mathcal{E}$ of Lemma \ref{lemma: J isogenous to P times E}.
Via the isomorphism $\frak{h}^3(\mathcal{J}/S)\simeq \bigwedge^3 \frak{h}^1(\mathcal{J}/S)$ \cite{kunnemann} and similarly for $\mathcal{P}$ and $\mathcal{E}$, we obtain a $\mu_6$-equivariant decomposition 
\begin{align}\label{decomposotion h^3(J)}
\frak{h}^3(\mathcal{J}/S) = \frak{h}^3(\mathcal{P}/S)  \oplus (\frak{h}^2(\mathcal{P}/S)\otimes \frak{h}^1(\mathcal{E}/S))\oplus (\frak{h}^1(\mathcal{P}/S)\otimes \frak{h}^2(\mathcal{E}/S)).
\end{align}
\begin{lemma}\label{lemma: mu6-invs h^3 controlled by h^2(P)}
    $\frak{h}^3(\mathcal{J}/S)^{\mu_6} =(\frak{h}^2(\mathcal{P}/S)\otimes \frak{h}^1(\mathcal{E}/S))^{\mu_3}$.
\end{lemma}
\begin{proof}
    By definition of $\mathcal{P}$ and $\mathcal{E}$ in \S\ref{subsec: prym varieties}, if $\mu_2(S) = \{1,\tau\}$ then $\tau$ acts as $(-1)$ on $\mathcal{P}$ and as the identity on $\mathcal{E}$. 
    Since $[^t\Gamma_{(n)}]$ acts as $n^i$ on $\frak{h}^i$ of an abelian scheme, we see that $\tau$ acts on $\frak{h}^i(\mathcal{P}/S)\otimes \frak{h}^{3-i}(\mathcal{E}/S)$ as $-1$ for $i=1,3$ and as the identity for $i=2$.
    It follows that $(\frak{h}^i(\mathcal{P}/S)\otimes \frak{h}^{3-i}(\mathcal{E}/S))^{\mu_6} = (\frak{h}^i(\mathcal{P}/S)\otimes \frak{h}^{3-i}(\mathcal{E}/S))^{\mu_2}=0$ if $i=1,3$ and $(\frak{h}^2(\mathcal{P}/S)\otimes\frak{h}^1(\mathcal{E}/S))^{\mu_6} = (\frak{h}^2(\mathcal{P}/S)\otimes\frak{h}^1(\mathcal{E}/S))^{\mu_3}$, so we conclude by \eqref{decomposotion h^3(J)}.
\end{proof}

It remains to determine $(\frak{h}^2(\mathcal{P}/S)\otimes \frak{h}^1(\mathcal{E}/S))^{\mu_3}$. 
To this end, we decompose $\frak{h}^2(\mathcal{P}/S)$ further using results of Moonen \cite{Moonen} and the quaternionic action on $\mathcal{P}$.

Recall the finite \'etale cover $\tilde{S}\rightarrow S$ from \S\ref{subsec: galois action NS group} given by adjoining a sixth root of unity and a sixth root of the discriminant $\Delta$.
This cover has Galois group $G = \Aut(\tilde{S}/S)\simeq D_6$.
Also recall from Theorem \ref{thm: Neron Severi action} the $G$-representation $\Hom^{\sym}(\mathcal{P}_{\tilde{S}},\mathcal{P}^{\vee}_{\tilde{S}})\otimes \Q$ of self-dual homomorphisms; let $\NS_{\mathcal{P}}$ be the corresponding Artin motive in $\mathsf{Mot}(S)$ under the equivalence of Lemma \ref{lemma: artin motives are local systems with finite monodromy}.
The $\mu_3$-action on $\mathcal{P}$ induces via pullback a $\mu_3$-action on $\NS_{\mathcal{P}}$ which is described in Lemma \ref{lemma: description mu3-action on NS(P)}.

\begin{proposition}\label{proposition: moonen decomposition h2}
    There exists a $\mu_3$-stable decomposition 
    \begin{align}\label{equation: moonen decomposition h2}
        \frak{h}^2(\mathcal{P}/S) = \frak{h}_{\mathrm{alg}}^2(\mathcal{P}/S)\oplus \frak{h}^2_{\mathrm{tr}}(\mathcal{P}/S)
    \end{align}
    in $\mathsf{Mot}(S)$ with the following two properties:
    \begin{enumerate}
        \item there is a $\mu_3$-equivariant isomorphism $\frak{h}_{\mathrm{alg}}^2(\mathcal{P}/S)\simeq \NS_{\mathcal{P}}(-1)$; 
        \item The $\mu_3$-action on $\frak{h}_{\mathrm{tr}}^2(\mathcal{P})$ is trivial.
    \end{enumerate}
\end{proposition}
\begin{proof} 
    By Corollary \ref{corollary: descending direct summands of motives}, it suffices to find a $G$-stable decomposition of $\frak{h}^2(\mathcal{P}_{\tilde{S}}/\tilde{S})$ with properties analogous to $(1)$ and $(2)$.
    Theorem \ref{thm: Neron Severi action} shows that $B=\End^0(\mathcal{P}_{\tilde{S}})$ is a nonsplit quaternion algebra.
    By functoriality we obtain a homomorphism $B^{\mathrm{op}}\rightarrow \End(\frak{h}^1(\mathcal{P}_{\tilde{S}}/\tilde{S}))$, where $B^{\mathrm{op}}$ is the opposite algebra of $B$.
    We identify $B^{\mathrm{op}}$ with $B$ using the canonical involution on $B$.
    Moonen \cite[Theorem 7.2]{Moonen} has shown that this $B$-action on $\frak{h}^1(\mathcal{P}_{\tilde{S}}/\tilde{S})$ can be used to refine the Chow--Künneth decomposition of $\frak{h}(\mathcal{P}_{\tilde{S}}/\tilde{S})$.
    Our specific situation is explained in detail in \cite[\S8.2]{Moonen}: let
    \begin{align}\label{equation: moonen decomposition over Stilde}
    \frak{h}^2(\mathcal{P}_{\tilde{S}}/\tilde{S}) = R^{(2,0)} \oplus R^{(1,1)}
    \end{align}
    be the decomposition described in \cite[Bottom of p.\ 104]{Moonen}.
    There exist algebraic representations $\rho_1, \rho_2$ of $B^{\times}$ such that $\Hom(M,R^{(2,0)})$ is $\rho_1$-isotypic and $\Hom(M,R^{(1,1)})$ is $\rho_2$-isotypic for every motive $M$ in $\mathsf{Mot}(S)$. 
    The representation $\rho_1$ is a twist of the conjugation action of $B^{\times}$ on trace zero elements in $B$, and the representation $\rho_2$ is the norm map $\text{Nm}\colon B^{\times}\rightarrow \Q^{\times}$.
    The $\mu_3$-action on $\mathcal{P}_{\tilde{S}}$ is induced from restricting the $B^{\times}$-action to a quadratic field $\Q(\omega)^{\times}\subset B^{\times}$ and further restricting to its subgroup $\mu_3 = \langle \omega\rangle \subset \Q(\omega)^{\times}$.
    Since $\text{Nm}$ is trivial on $\mu_3$, it follows that $\mu_3$ acts trivially on $R^{(1,1)}$.
    By construction of $R^{(2,0)}$ \cite[\S6.4]{Moonen}, there is a canonical isomorphism $\alpha\colon \mathbb{L}\otimes (\Hom^{\sym}(\mathcal{P}_{\tilde{S}}, \mathcal{P}_{\tilde{S}}^{\vee})\otimes \Q) \xrightarrow{\sim} R^{(2,0)}$.
    Since $G$ preserves the isomorphism classes of $\rho_1$ and $\rho_2$, the decomposition \eqref{equation: moonen decomposition over Stilde} is $G$-stable. 
    Therefore by Corollary \ref{corollary: descending direct summands of motives}, $R^{(2,0)}\oplus R^{(1,1)}$ descends to a decomposition $\frak{h}_{\mathrm{alg}}^2(\mathcal{P}/S)\oplus \frak{h}^2_{\mathrm{tr}}(\mathcal{P}/S)$.
    Since the isomorphism $\alpha$ is $G$-equivariant, it descends to an isomorphism $\NS_{\mathcal{P}}(-1) \simeq \frak{h}^2_{\mathrm{alg}}(\mathcal{P}/S)$, as desired.
\end{proof}

For the remainder of the paper we will fix a decomposition \eqref{equation: moonen decomposition h2}. 
 It is in fact unique, but we will not use this.

\begin{remark}
{\em
    Let $\bar{K}$ be the algebraic closure of the function field of $S$.
    Then the base change of the decomposition \eqref{equation: moonen decomposition h2} along $\Spec (\bar{K})\rightarrow S$ is the decomposition of $\frak{h}^2(\mathcal{P}_{\bar K})$ into its algebraic and transcendental part \cite{KahnMurrePedrini-transcendentalpartmotivesurface}, explaining our choice of notation.}
\end{remark}

\begin{corollary}\label{corollary: mu6-invs h^2(P) controlled by h^2alg}
      $(\frak{h}^2(\mathcal{P}/S)\otimes \frak{h}^1(\mathcal{E}/S))^{\mu_3}\simeq (\NS_{\mathcal{P}}\otimes \frak{h}^1(\mathcal{E}/S))^{\mu_3}(-1)$
\end{corollary}
\begin{proof}
    By Proposition \ref{proposition: moonen decomposition h2}, $\mu_3$ acts trivially on $\frak{h}^2_{\mathrm{tr}}(\mathcal{P}/S)$.
    Since the $\mu_3$-action on $\mathcal{E}$ is faithful, $\frak{h}^1(\mathcal{E}/S)^{\mu_3}=0$.
    We conclude that $(\frak{h}^2_{\mathrm{tr}}(\mathcal{P}/S)\otimes \frak{h}^1(\mathcal{E}/S)))^{\mu_3}=0$ and so $(\frak{h}^2(\mathcal{P}/S)\otimes \frak{h}^1(\mathcal{E}/S))^{\mu_3} =(\frak{h}_{\mathrm{alg}}^2(\mathcal{P}/S)\otimes \frak{h}^1(\mathcal{E}/S))^{\mu_3}  \simeq (\NS_{\mathcal{P}}\otimes \frak{h}^1(\mathcal{E}/S))^{\mu_3}(-1)$, again by Proposition \ref{proposition: moonen decomposition h2}.
\end{proof}

To analyze the piece $\NS_\mathcal{P}\otimes \frak{h}^1(\mathcal{E}/S)$, we define and study an auxiliary abelian scheme $\mathcal{E}\otimes \mathrm{std}$.
Recall from \S\ref{subsec: galois action NS group} the finite \'etale cover $\tilde{S}\rightarrow S$ with Galois group $G \simeq D_6$ and its two-dimensional representation $\mathrm{std}\colon G\rightarrow \GL_2(\Q)$, which in fact lands in $\GL_2(\Z)$.
Let $\mathcal{E} \otimes \mathrm{std}$ be the (up to isomorphism) unique abelian scheme over $S$ that is isomorphic to $\mathcal{E}^2$ over $\tilde{S}$ with descent data given by $\mathrm{std}\colon G\rightarrow \GL_2(\Z)\subset\End(\mathcal{E}^2)$. (This is the analogue of the construction of \cite{MazurRubinSilverberg} in the relative setting.)
\begin{lemma}\label{lemma: E otimes std is isogenous to two elliptic curves}
    There is an isogeny $\mathcal{E}\otimes \mathrm{std}\sim \mathcal{E}^{\Delta} \times \mathcal{E}^{\Delta^{-1}}$.
\end{lemma}
\begin{proof}
    Let $K = \Q(a,b)$ be the function field of $S$ and denote the base change of $\mathcal{E}$ to $K$ by $E$.
    Proposition \ref{prop: abelian schemes normal bases} shows that it suffices to prove the lemma over $K$.
    Let $\tilde{K} = K(\omega,\sqrt[6]{\Delta})$ denote the function field of $\tilde{S}$. 
    For $d\in K^{\times}$, write $E^d:y^2 = x^3 + 16(a^2-4b)d$ for the sextic twist of $E$ by $d$.
    Then $E \otimes \mathrm{std}$ is an isogeny factor of \[\Res^{\tilde K}_KE \sim \Res^{K(\sqrt[6]{\Delta})}_K( E^2) \sim \left(\Res^{K(\sqrt[6]{\Delta})}_KE\right)^2 \sim \left(\prod_{i = 0}^5 E^{\Delta^i}\right)^2. \] 
    Indeed: the first isogeny follows from the fact that $E$ is isogenous to its own $(-3)$-quadratic twist (this being true for any $j$-invariant $0$ elliptic curve), the second is a formal property of restriction of scalars, and the third is the isogeny decomposition that corresponds (by base change) to the six characters of $H = \Gal(\tilde K/K(\omega))$. Since $\mathrm{std}|_H$ is the sum of the two characters of order $6$, we must have 
    \[E \otimes \mathrm{std} \sim E^\Delta \times E^{\Delta^5} \simeq E^\Delta \times E^{\Delta^{-1}},\]
    as claimed.
\end{proof}

\begin{lemma}\label{lemma: mu6-invs NS gives Delta twist}
    $(\NS_\mathcal{P}\otimes \frak{h}^1(\mathcal{E}/S))^{\mu_3} \simeq \frak{h}^1(\mathcal{E}^{\Delta}/S)$.
\end{lemma}

\begin{proof}
Use Theorem \ref{thm: Neron Severi action} to decompose the $G$-representation $\NS_{\mathcal{P}}$ as $\mathrm{triv}\oplus \mathrm{std}$. 
By Lemma \ref{lemma: description mu3-action on NS(P)}, $\mu_3$ acts trivially on $\mathrm{triv}$ and as a direct sum of two nontrivial characters on $\mathrm{std}\otimes \Qbar$.
The $\mu_3$-actions on $\mathcal{E}$ and $\mathrm{std}$ combine to a $\mu_3$-action on $\mathcal{E}\otimes \mathrm{std}$.
Write $\mathcal{A} := (\mathcal{E}\otimes \mathrm{std})^{\mu_3, \circ }$.
We then have 
\begin{equation}\label{eq: fixed point computation}
(\NS_\mathcal{P}\otimes \frak{h}^1(\mathcal{E}/S))^{\mu_3} \simeq ( \mathrm{std}\otimes \frak{h}^1(\mathcal{E}/S))^{\mu_3} \simeq (\frak{h}^1((\mathcal{E}\otimes \mathrm{std})/S))^{\mu_3}\simeq  \frak{h}^1(\mathcal{A}/S),
\end{equation}
where the first isomorphism follows from the fact that $\frak{h}^1(\mathcal{E}/S)^{\mu_3}=0$, and the last two follow from Lemma \ref{lemma: descending morphisms of motives} and Proposition \ref{proposition: AVs up to isogeny same as h^1}.
To prove the lemma, it remains to show that the abelian scheme $\mathcal{A}$ is isogenous to $\mathcal{E}^\Delta$. 

Let $\mathcal{A}'$ be the quotient abelian scheme $(\mathcal{E} \otimes \mathrm{std})/ \mathcal{A}$. 
A tangent space calculation (using the $\mu_3$-action on $\mathrm{std}$ and $\mathcal{E}$) shows that $\mathcal{A}\rightarrow S$ has relative dimension $1$, hence so does the quotient $\mathcal{A}' \rightarrow S$.
Given $d\in \mathcal{O}(S)^{\times}$ write $\mathcal{E}^d\colon y^2 = x^3 + 16(a^2-4b)d$ for the sextic twist of $\mathcal{E}$ by $d$.
By complete reducibility and Lemma \ref{lemma: E otimes std is isogenous to two elliptic curves}, $\mathcal{A}$ is isogenous to either
$\mathcal{E}^{\Delta} \colon y^2 = x^3 + 4b(a^2 - 4b)^2$ or $\mathcal{E}^{\Delta^{-1}} \colon y^2 = x^3 + b^{-1}$. 
We will exclude the latter possibility using a rather indirect argument.

Fix a unit $d\in \mathcal{O}(S)^{\times}$, consider the isomorphism $\lambda_d\colon S\rightarrow S$ mapping $(a,b)$ to $(da,d^2b)$ and let $\mathcal{J}^d = \lambda_d^*\mathcal{J}$.
Then $\mathcal{J}$ and $\mathcal{J}^d$ are sextic twists in the sense that they become isomorphic over the $\mu_6$-cover $f_d\colon S_d \rightarrow S$ given by adjoining a sixth root $\delta$ of $d$.
There exists an isomorphism $f_d^*\mathcal{J} \xrightarrow{\sim} f_d^*\mathcal{J}^d$ of abelian schemes over $S_d$ induced by the isomorphism $(x,y)\mapsto (\delta^3x,\delta^4y)$ between $\mathcal{C}$ and $\lambda_d^*\mathcal{C}$.
This in turn induces an isomorphism $\frak{h}^3(f_d^*\mathcal{J}^d/S_d)\xrightarrow{\sim}\frak{h}^3(f_d^*\mathcal{J}/S_d)$.
Lemma \ref{lemma: descending morphisms of motives} shows that the induced morphism on $\mu_6$-fixed points descends to an isomorphism $\frak{h}^3(\mathcal{J}/S)^{\mu_6}\simeq \frak{h}^3(\mathcal{J}^d/S)^{\mu_6}$ over $S$.

Now assume, for the sake of contradiction, that $\mathcal{A}\sim\mathcal{E}^{\Delta^{-1}}$.
Then $\frak{h}^3(\mathcal{J}/S)^{\mu_6} \simeq \frak{h}^1(\mathcal{E}^{\Delta^{-1}}/S)(-1)$ by Lemma \ref{lemma: mu6-invs h^3 controlled by h^2(P)}, Corollary \ref{corollary: mu6-invs h^2(P) controlled by h^2alg} and \eqref{eq: fixed point computation}. Pulling back this isomorphism along $\lambda_d$ and using the fact that $\lambda^*_d \mathcal{E}^{\Delta^{-1}} \simeq \mathcal{E}^{d^{-2}\Delta^{-1}}$, we get an isomorphism $\frak{h}^3(\mathcal{J}^d/S)^{\mu_6} \simeq \frak{h}^1(\mathcal{E}^{d^{-2}\Delta^{-1}}/S)(-1)$.
Combining the last three isomorphisms shows that $\frak{h}^1(\mathcal{E}^{\Delta^{-1}}/S) \simeq \frak{h}^1(\mathcal{E}^{d^{-2}\cdot \Delta^{-1}}/S)$. 
 By Proposition \ref{proposition: AVs up to isogeny same as h^1}, $\mathcal{E}^{\Delta^{-1}}$ and $\mathcal{E}^{d^{-2}\cdot \Delta^{-1}}$ are isogenous over $S$ for all such $d$.
 This is a contradiction, since setting $s= (a,b) = (1,1) \in S(\Q)$ and $d=3^{-1} \in \mathcal{O}(S)^{\times}$ produces two curves $\mathcal{E}^{\Delta^{-1}}_{(a,b)}\colon y^2 = x^3+1$ and $\mathcal{E}^{9\Delta^{-1}}_{(a,b)}\colon y^2 = x^3+9$ that are not isogenous over $\Q$ (by counting $\F_7$-points, for example).
\end{proof}

\begin{proof}[Proof of Theorem \ref{theorem: iso relative chow motives}]
    We have
    \begin{align*}
        \frak{h}^3(\mathcal{J}/S)^{\mu_6} = (\frak{h}^2(\mathcal{P}/S)\otimes \frak{h}^1(\mathcal{E}/S))^{\mu_3} \simeq (\NS_{\mathcal{P}}\otimes \frak{h}^1(\mathcal{E}/S))^{\mu_3}(-1) 
        \simeq \frak{h}^1(\mathcal{E}^{\Delta}/S)(-1),
    \end{align*}
    using Lemma \ref{lemma: mu6-invs h^3 controlled by h^2(P)}, Corollary \ref{corollary: mu6-invs h^2(P) controlled by h^2alg} and Lemma \ref{lemma: mu6-invs NS gives Delta twist} respectively.
\end{proof}

\subsection{Chow groups}\label{subsec: chow groups}

Our next goal is to translate the isomorphism of Theorem \ref{theorem: iso relative chow motives} into a concrete statement about Chow groups. 
We will compute $\CH^2$ of $\mathfrak{h}^3(\mathcal{J}/S)^{\mu_6}$ and $\mathfrak{h}^1(\mathcal{E}^{\Delta}/S)(-1)$ (in the sense of \S\ref{subsec: chow groups of motives}) below.
To this end we will use the Beauville decomposition of the Chow groups of an abelian variety \cite{Beauville-surlanneadechow}, extended by Deninger--Murre to the relative setting \cite[Theorem 2.19]{DeningerMurre}.

\begin{proposition}[Beauville decomposition]
    Let $A/S$ be an abelian scheme and $p\in \Z_{\geq 0}$.
    For $i\in \Z$ let 
    \[\CH^p_{(i)}(A; \Q) := \{ \alpha \in \CH^p(A;\Q) \mid (n)^*\alpha = n^{2p-i}\alpha \text{ for all }n\in \Z\}.
    \]
    Then $\CH^p_{(i)}(A;\Q)$ is nonzero for only finitely many $i$ and we have a direct sum decomposition
    \begin{align}\label{equation: beauville decomposition}
    \CH^p(A;\Q) = \bigoplus_{i\in \Z} \CH^p_{(i)}(A;\Q).    
    \end{align}
\end{proposition}

\begin{proposition}\label{prop: chow groups}
    Let $T$ be an adequate scheme and $T\rightarrow S$ a morphism. 
    Then there are isomorphisms:
    \begin{enumerate}
        \item $\CH^2( \frak{h}^1(\mathcal{E}^{\Delta}_T/T)(-1)) = \mathcal{E}^{\Delta}(T) \otimes \Q$.
        \item $\CH^2(\frak{h}^3(\mathcal{J}_T/T)^{\mu_6}) = \CH^2_{(1)}(\mathcal{J}_T;\Q)^{\mu_6}$.
    \end{enumerate}
    Moreover if $T'\rightarrow S$ is another such morphism and if $f\colon T'\rightarrow T$ is an $S$-morphism that admits pullbacks (see \S\ref{subsec: base changing motives}) then the above isomorphisms are compatible with base change along $f$.
\end{proposition}
\begin{proof}
    \begin{enumerate}
        \item We have $\CH^2( \frak{h}^1(\mathcal{E}^{\Delta}_T/T)(-1))=\CH^1( \frak{h}^1(\mathcal{E}^{\Delta}_T/T))$. We will prove the stronger statement that if $X\rightarrow T$ is an abelian scheme of relative dimension $1$ then $\CH^1(\frak{h}^1(X/T)) = X(T)\otimes \Q$.
        This is the relative version of a classical curve computation \cite[Theorem 2.7.2]{MurreNagelPeters}; we briefly sketch the details.
        Write $\Gamma_e\subset X\times_T X$ for the graph of the zero section $e\colon T\rightarrow X$.
        Then $\frak{h}^1(X/T)=(X/T, [\Delta_{X/T}]-[\Gamma_e]-[^t\Gamma_e],0)$.
        Using the equivalence between Cartier divisors and line bundles and unwinding the definition of $\CH^1(\frak{h}^1(X/T))$, we see that it is exactly the $\Q$-tensor product of the set of isomorphism classes of line bundles on $X$, fiberwise of degree zero and rigidified along the zero section. 
        In other words, $\CH^1(\frak{h}^1(X/T)) = \mathbf{Pic}^0_{X/T}(T)\otimes \Q$, where $\mathbf{Pic}^0_{X/T}$ denotes the identity component of the Picard functor \cite[\S8.1, Proposition 4]{BLR-neronmodels}.
        Since $X$ is a relative elliptic curve, the map $P\mapsto [P]-[e]$ induces an isomorphism $X\simeq  \mathbf{Pic}^0_{X/T}$.
        \item By definition of the Chow--Künneth decomposition of Theorem \ref{theorem: deninger--Murre projectors} and the Beauville decomposition \eqref{equation: beauville decomposition}, $\CH^p(\frak{h}^i(A/S)) = \CH^p_{(2p-i)}(A)$ for every abelian scheme $A/S$ and $i,p \in \Z$.
        Therefore $\CH^2(\frak{h}^3(\mathcal{J}_T/T)) = \CH^2_{(1)}(\mathcal{J}_T;\Q)$.
        Since $\CH^p(M^G) = \CH^p(M)^G$ for every finite group $G$ acting on a motive $M$ (see \S\ref{subsec: chow groups of motives}), the result follows from taking $\mu_6$-fixed points.
    \end{enumerate}
    The statement about compatibility with pullbacks follows from the construction of the isomorphisms.
\end{proof}

\begin{corollary}\label{cor: isomorphism of Chow groups}
    Let $T$ be an adequate scheme and $T\rightarrow S$ a morphism that admits pullbacks (see \S\ref{subsec: base changing motives}). 
    Then there exists an isomorphism of $\Q$-vector spaces 
    \[
    \Phi_T\colon \CH^2_{(1)}(\mathcal{J}_T;\Q)^{\mu_6} \xrightarrow{\sim} \mathcal{E}^{\Delta}(T)\otimes \Q\]
    with the following property:
    If $T'\rightarrow S$ is another such morphism and $f\colon T'\rightarrow T$ a morphism of $S$-schemes that admits pullbacks then the following diagram is commutative:
\begin{center}
\begin{tikzcd}
\CH^2_{(1)}(\mathcal{J}_T;\Q)^{\mu_6} \arrow[r, "\Phi_T"] \arrow[d, "f^*"] & \mathcal{E}^{\Delta}(T)\otimes \Q \arrow[d, "f^*"] \\
\CH^2_{(1)}(\mathcal{J}_{T'};\Q)^{\mu_6}\arrow[r, "\Phi_{T'}"]                 & \mathcal{E}^{\Delta}(T')\otimes \Q                
\end{tikzcd}
\end{center}
\end{corollary}
\begin{proof}
    Choose an isomorphism $\phi\colon \frak{h}^3(\mathcal{J}/S)^{\mu_6} \xrightarrow{\sim} \frak{h}^1(\mathcal{E}^{\Delta}/S)$ using Theorem \ref{theorem: iso relative chow motives}.
    This defines an isomorphism $\phi_T$ in $\mathsf{Mot}(T)$ after base change, see \S\ref{subsec: base changing motives}. 
    Taking $\CH^2$ of this isomorphism and using Proposition \ref{prop: chow groups}, we get the desired isomorphism $\Phi_T$.
\end{proof}

\subsection{The image of the Ceresa cycle}\label{subsec: image of ceresa cycle}

The defining equation \eqref{equation: bielliptic picard body of text} of $\mathcal{C}$ shows there exists a unique section $\infty\colon S\rightarrow \mathcal{C}$ at infinity.
We view $\mathcal{C}$ as a closed subscheme of $\mathcal{J}$ using the Abel--Jacobi map based at $\infty$.
We define the \define{universal Ceresa cycle} to be the cycle
\[\kappa_{\infty}(\mathcal{C}) := [\mathcal{C}] - (-1)^*[\mathcal{C}] \in \CH^2(\mathcal{J}).\]
For every field-valued point $s\colon \Spec(k) \rightarrow S$, the pullback $s^*(\kappa_{\infty}(\mathcal{C})) \in \CH^2(\mathcal{J}_s)$ equals the Ceresa cycle $\kappa_{\infty}(\mathcal{C}_s)$ of the curve $\mathcal{C}_s$ over $k$ with respect to the basepoint at infinity defined in \S\ref{subsec: ceresa and GS cycles def}.
We will work with $\Q$-coefficients, so let $\kappa(\mathcal{C})$ be the image of $\kappa_{\infty}(\mathcal{C})$ in $\CH^2(\mathcal{J};\Q)$, and for each $s\in S(k)$ let $\kappa(\mathcal{C}_s)$ be the image of $\kappa_{\infty}(\mathcal{C}_s)$ in $\CH^2(\mathcal{J}_s;\Q)$.
The next lemma shows that, at least fiberwise, $\CH^2$ of $\frak{h}^3(\mathcal{J}/S)^{\mu_6}$ naturally contains the Ceresa cycle.

\begin{lemma}\label{lem: ceresa specializations}
    For every field $k$ and $s\in S(k)$, we have $\kappa(\mathcal{C}_s) \in \CH^2_{(1)}(\mathcal{J}_s;\Q)^{\mu_6} = \CH^2(\frak{h}^3(\mathcal{J}_s)^{\mu_6})$.
\end{lemma}
\begin{proof}
    Write $C = \mathcal{C}_s$ and $J = \mathcal{J}_s$. 
    Since $J$ is defined over a field, Beauville has shown \cite[Th\'eor\`eme and Proposition 3(a)]{Beauville-surlanneadechow} that $\CH^2_{(1)}(J;\Q)$ is zero unless $i\in \{0,1,2\}$.
    So using \eqref{equation: beauville decomposition} we may write $[C] = [C]_0+[C]_1+[C]_2$ in $\CH^2(J;\Q)$, where $(n)^*$ acts on $[C]_i$ as $n^{4-i}$.
    Therefore $\kappa(C) = [C]-(-1)^*[C] = ([C]_0+[C]_1+[C]_2)-([C]_0-[C]_1+[C]_2)=2[C]_1 \in \CH^2_{(1)}(J;\Q)$.
    Since $\kappa(C)$ is also $\mu_6$-invariant, we conclude the lemma.
\end{proof}

Let $\kappa(\mathcal{C})_{(1)}$ be the component in $\CH^2_{(1)}(\mathcal{J})$ of $\kappa(\mathcal{C})$ in the Beauville decomposition \eqref{equation: beauville decomposition} of $\CH^2(\mathcal{J};\Q)$.
(It might very well be that $\kappa(\mathcal{C}) = \kappa(\mathcal{C})_{(1)}$, but this is not needed in our analysis.)
Since $\kappa(\mathcal{C})$ is $\mu_6$-invariant, $\kappa(\mathcal{C})_{(1)} \in \CH^2_{(1)}(\mathcal{J})^{\mu_6} = \CH^2(\frak{h}^3(\mathcal{J}/S)^{\mu_6})$.

Choose for the remainder of the paper a collection of isomorphisms $\Phi_T$ satisfying the conclusions of Corollary \ref{cor: isomorphism of Chow groups}.
Define
\[
\sigma_0:= \Phi_S(\kappa(\mathcal{C})_{(1)}) \in \mathcal{E}^{\Delta}(S)\otimes \Q.
\]
Let $\text{den}(\sigma_0)$ be the smallest positive integer $d$ such that $d\cdot \sigma_0 \in \mathcal{E}^{\Delta}(S)$.
Let $\sigma = \text{den}(\sigma_0) \cdot \sigma_0$, a section of the abelian scheme $\mathcal{E}^{\Delta}\rightarrow S$.

\begin{corollary}\label{corollary: ceresa torsion iff section torsion}
    For every field $k$ and $s\in S(k)$, $\sigma(s) \in \mathcal{E}^{\Delta}_s(k)$ is torsion if and only if $\kappa_{\infty}(\mathcal{C}_s)\in \CH^2(\mathcal{J}_s)$ is torsion.
\end{corollary}
\begin{proof}
    Follows from Corollary \ref{cor: isomorphism of Chow groups} and Lemma \ref{lem: ceresa specializations}.
\end{proof}

For technical reasons when relating the torsion orders of $\kappa_{\infty}(\mathcal{C}_s)$ and $\sigma(s)$, it will be useful to have an integral lift of a multiple of $\kappa(\mathcal{C})_{(1)}$, in the following sense:

\begin{lemma}\label{lemma: integral lift kappa_1}
    There exists a class $\kappa_{\Z} \in \CH^2(\mathcal{J})$ and an integer $N\geq 1$ such that the image of $\kappa_{\Z}$ in $\CH^2(\mathcal{J};\Q)$ is $N\cdot \kappa(\mathcal{C})_{(1)}$ and such that $s^*(\kappa_{\Z}) = N\cdot \kappa_{\infty}(\mathcal{C}_s)$ for every field-valued point $s\colon \Spec(k)\rightarrow S$.
\end{lemma}
\begin{proof}
    Write $\kappa(\mathcal{C}) = \kappa(\mathcal{C})_{(1)} + \kappa'$ in $\CH^2(\mathcal{J};\Q)$, where $\kappa'$ is the sum of the other components in the Beauville decomposition \eqref{equation: beauville decomposition}.
    There exist an integer $N_1\geq 1$ and classes $\alpha, \beta \in \CH^2(\mathcal{J})$ whose images in $\CH^2(\mathcal{J};\Q)$ are $N_1 \kappa(\mathcal{C})_{(1)}$ and $N_1\kappa'$ respectively.
    Since the class $N_1 \kappa(\mathcal{C}) -\alpha- \beta$ is zero in $\CH^2(\mathcal{J};\Q)$, it must be torsion in $\CH^2(\mathcal{J})$, so there exists an integer $N_2\geq 1$ such that $N_1N_2\kappa(\mathcal{C}) =N_2 \alpha + N_2\beta $.

    We claim that there exists an integer $N_3\geq 1$ such that $N_3\cdot s^*(\beta)=0$ for every field-valued point $s\colon \Spec(k) \rightarrow S$.
    Lemma \ref{lem: ceresa specializations} shows that $s^*(\beta)$ is torsion in $\CH^2(\mathcal{J}_s)$ for every such $s$.
    So if $s= \eta$ is the generic point of $S$, there exists an $M\geq 1$ such that $M\cdot \eta^*(\beta) = 0$.
    By spreading out, there exists an open $U\subset S$ such that $M\cdot s^*(\beta)=0$ for all $s\in U(k)$.
    Repeatedly applying this argument to the generic points of the irreducible components of $S\setminus U$ proves the claim.

    The lemma follows by taking $\kappa_{\Z} = N_2N_3\alpha$ and $N = N_1N_2N_3$.    
\end{proof}

\subsection{Identifying $\sigma$}\label{subsec: identifying sigma}

Our next goal is to explicitly identify the section $\sigma$ of $\mathcal{E}^{\Delta}\rightarrow S$, at least up to integer multiples. 
The equation $y^2 = x^3 + 4b(a^2-4b)^2$ allows us to define $\G_m$-actions on $\mathcal{E}^{\Delta}$ and $S\subset \A^2$ by the formulas $\lambda \cdot (x,y) = (\lambda^2x, \lambda^3y)$ and $\lambda\cdot (a,b) = (\lambda a, \lambda^2 b)$.
With these actions, the morphism $\mathcal{E}^{\Delta}\rightarrow S$ is $\G_m$-equivariant. 
Write $\mathrm{MW}^{\G_m}(\mathcal{E}^{\Delta}/S)$ for the set of $\G_m$-equivariant sections of $\mathcal{E}^{\Delta}\rightarrow S$.
It is an abelian group, the $\G_m$-equivariant Mordell--Weil group of $\mathcal{E}^{\Delta}/S$.

\begin{lemma}\label{lem: elliptic surface has rank 1}
    $\mathrm{MW}^{\G_m}(\mathcal{E}^{\Delta}/S)$ is free of rank $1$, generated by the section $\mathcal{Q} := ((a^2 - 4b),a(a^2 - 4b))$.
\end{lemma}
\begin{proof}
    Let $T$ be the closed subscheme of $S$ given by setting $a=1$.
    Restricting sections to $T$ induces an injection $\mathrm{MW}^{\G_m}(\mathcal{E}^{\Delta}/S)\hookrightarrow \mathcal{E}^{\Delta}(T)$, so it suffices to prove that $\mathcal{Q}|_T$ generates the Mordell--Weil group of the elliptic surface $\mathcal{E}^{\Delta}|_T \rightarrow T$.

    Let $\pi\colon \mathcal{X} \rightarrow\P^1_{\Q}$ be the minimal regular model of $\mathcal{E}^{\Delta}_T \rightarrow T = \P^1_{\Q}\setminus \{0,1/4,\infty\}$.
    Using Tate's algorithm, we see that the three singular fibers above $b=0,1/4,\infty$ have Kodaira type $\II, \IV$ and $\IZS$ respectively.
    By \cite[Lemma 7.8]{SchuttShioda-ellipticsurfaces}, there are no nonzero torsion sections of $\pi$.
    By the Shioda--Tate formula \cite[Theorem 6.3, Proposition 6.6 and \S8.8]{SchuttShioda-ellipticsurfaces}, the Mordell--Weil group of $\mathcal{X}_{\bar{\Q}}\rightarrow \P^1_{\bar{\Q}}$ is free of rank $2$.
    The $\mu_3$-action on $\mathcal{E}^{\Delta}$ shows that it is also a $\Z[\omega]$-module, necessarily of rank $1$.
    Taking Galois invariants, it follows that the Mordell--Weil group of $\mathcal{X}\rightarrow \P^1_{\Q}$ is free of rank $1$ over $\Z$.
    To show that it is generated by (the closure of) $\mathcal{Q}|_T$, it suffices to find a single $b\in T(\Q)$ for which $\mathcal{Q}_b$ is primitive in $\mathcal{E}^{\Delta}_b(\Q)$, i.e.\ not divisible by $n$ for every $n\geq 2$.
    One may \href{https://www.lmfdb.org/EllipticCurve/Q/972/a/2}{check} that this holds for $b=1$.
\end{proof}

\begin{theorem}\label{thm: ceresa section is multiple of branch point}
    There exists an integer $N\geq 1$ such that $\sigma = N\cdot \mathcal{Q}$.
\end{theorem}

\begin{proof}
    The section $\sigma \in \mathcal{E}^{\Delta}(S)$ is $\G_m$-equivariant.
    Moreover by our concrete computation with $y^3 = x^4+x^2+1$ in Lemma \ref{lem:section is not 0} (combined with Corollary \ref{corollary: ceresa torsion iff section torsion}), it has a nonzero specialization, so it is not identically zero.
    We conclude using Lemma \ref{lem: elliptic surface has rank 1}.
\end{proof}

\subsection{Proof of the main theorem}\label{subsec: the main theorem}

Putting everything together, we prove the following theorem which implies Theorem \ref{thm:main} of the introduction.
\begin{theorem}\label{theorem: ceresa torsion statement in terms of Edelta}
    Let $k$ be a field of characteristic $\neq 2,3$ and let $C/k$ be a smooth projective curve with equation $y^3 = x^4+ax^2+b$ and Jacobian $J$.
    Then the Ceresa cycle $\kappa_{\infty}(C) \in \CH^2(J)$ is torsion if and only if $\mathcal{Q}_{(a,b)} = ((a^2-4b), a(a^2-4b))$ is a torsion point of the elliptic curve $E^{\Delta}\colon y^2 = x^3 + 4b(a^2-4b)^2$.

    Moreover, the quotient of the torsion orders $\ord(\kappa_{\infty}(C))/\ord(\mathcal{Q}_{a,b})$, whenever it is defined, takes finitely many values when varying over all $k$ and $C$ as above.
\end{theorem}

\begin{proof}[Proof of Theorem \ref{theorem: ceresa torsion statement in terms of Edelta}]
    The curve $C$ is isomorphic to the pullback $\mathcal{C}_s$ of $\mathcal{C}\rightarrow S$ along the $k$-point $s = (a,b) \colon \Spec(k)\rightarrow S$.
    By Corollary \ref{corollary: ceresa torsion iff section torsion}, $\kappa_{\infty}(C)$ is torsion if and only if $\sigma(s)\in \mathcal{E}^{\Delta}_s(k)$ is torsion.
    By Theorem \ref{thm: ceresa section is multiple of branch point}, $\sigma(s)$ is a multiple of $\mathcal{Q}_s = \mathcal{Q}_{(a,b)}$, so $\sigma(s)$ is torsion if and only if $\mathcal{Q}_s = (a^2-4b,a(a^2-4b))$ is torsion. 
    Tracing through these equivalences and using Lemma \ref{lemma: integral lift kappa_1}, the quotients of the torsion orders take finitely many values at each step, so $\ord(\kappa_{\infty}(C))/\ord(\mathcal{Q}_{a,b})$ takes finitely many values as well.
\end{proof}

\begin{proof}[Proof of Theorem \ref{thm:main}]
    We take $(a,b)=(2t,1)$. A simple computation (using that $k = \overline{k}$) shows that there is an isomorphism from $E_{a,b}^\Delta$ to $\widehat{E} \colon y^2 = x^3 + 1$, sending $\mathcal{Q}_{(a,b)} = (4t^2-4,2t(4t^2-4))$ to $Q = (\sqrt[3]{t^2-1},t)$, for some choice of cube root. Thus $\kappa_{\infty}(C_t)$ is torsion if and only if $Q$ is torsion.
\end{proof}

\subsection{Corollaries of Theorem \ref{thm:main}}\label{subsec: proofs}

First we prove the classification of torsion Ceresa cycles over $\Q$ claimed in the introduction. We use the following classification of rational torsion points on the elliptic curves $\widehat{E}_d \colon y^2 = x^3 + d$ of $j$-invariant $0$. 

\begin{lemma}\label{lem:j-invariant 0 classification}
    Let $d \in \Q^\times$. Then
    \begin{enumerate}
    \item $\widehat{E}_d(\Q)_{\mathrm{tors}}$ is a subgroup of $\Z/6\Z$;
        \item $\widehat{E}_d(\Q)$ contains a point of order $2$ if and only if $d$ is a cube;  
        \item $\widehat{E}_d(\Q)$ contains a point of order $3$ if and only if $d$ is a square or $d = -2^43^3m^6$.
    \end{enumerate}
\end{lemma}

\begin{proof}
Since $\widehat{E}_d$ is a CM elliptic curve, it has potentially good reduction at every prime $p$. The prime-to-$p$ part of $\widehat{E}_d(\Q)_{\mathrm{tors}}$ therefore injects into $B(\F_p)$ for some $j$-invariant $0$ elliptic curve $B/\F_p$. Thus $(1)$ follows from the list of $\F_p$-points of $j$-invariant $0$ elliptic curves over $\F_3$ and $\F_5$. Parts $(2)$ and $(3)$ can be read off from the $2$- and $3$- division polynomials of $X_d$, which are $x^3 + d$ and $x(x^3 + 4d)$, respectively.   
\end{proof}

In the following proposition we use the notation of the introduction and denote the bielliptic Picard curve $C_{2t,1}$ by $C_t$.

\begin{proposition}\label{prop: Q torsion ceresa cycles}
    Suppose $C$ is a curve over $\Q$ that is $\Qbar$-isomorphic to some bielliptic Picard curve  $C_{a,b}$  and has torsion Ceresa cycle. Then $C$ is $\Qbar$-isomorphic to either $C_0$, $C_3$, or $C_{\sqrt{-3}}$.
\end{proposition}
\begin{proof}
    By Lemma \ref{lemma: j-invariant bielliptic picard curve field of definition}, $C$ is $\Qbar$-isomorphic to some $C_{a,b}$ with rational $j$-invariant, and the proof of that lemma shows that $a,b$ can be chosen to be in $\Q$.
    Since $C_{a,b}$ has torsion Ceresa cycle, the point $\mathcal{Q}_{a,b} = (a^2-4b,a(a^2-4b))$ is a $\Q$-rational torsion point on the $j$-invariant $0$ elliptic curve $E^\Delta_{a,b}:y^2 = x^3+4b(a^2-4b)^2$. It follows from Lemma \ref{lem:j-invariant 0 classification} that $\mathcal{Q}_{a,b}$ has order $2,3$, or $6$. 
    If $\mathcal{Q}_{a,b}$ has order $2$, then the $y$-coordinate $a(a^2-4b)^2$ of $\mathcal{Q}_{a,b}$ is zero, so since $\Delta_{a,b}\neq 0$ we have $a=0$. 
    Hence $C$ is isomorphic to $C_{0,b}$, which is $\Qbar$-isomorphic to $C_{0,1} = C_0$.
    Similarly, if $\mathcal{Q}_{a,b}$ has order $3$, we compute that $a^2+12b=0$ hence $C$ is $\Qbar$-isomorphic to $C_{\sqrt{-3}}$.
    If $\mathcal{Q}_{a,b}$ has order $6$, then Lemma \ref{lem:j-invariant 0 classification} shows that there exists an isomorphism between $E^{\Delta}_{a,b}$ and $y^2 = x^3+1$ that sends $\mathcal{Q}_{a,b}$ to $(2,3)$.
    A calculation shows that $a^2-36b=0$, in other words $C$ is $\Qbar$-isomorphic to $C_{3}$.
\end{proof}
\begin{proof}[Proof of Corollary \ref{cor: main infinitude}]
Follows from Proposition \ref{prop: Q torsion ceresa cycles}.
\end{proof}

\begin{proof}[Proof of Theorem \ref{thm: arbitrarily large order}]
The torsion specializations $\mathcal{Q}_s\in \mathcal{E}^{\Delta}_s(\C)$, where $s$ varies in $S(\C)$, are all defined over $\Qbar$ and have unbounded order.
We conclude by Theorem \ref{theorem: ceresa torsion statement in terms of Edelta}.
\end{proof}

Theorem \ref{thm: northcott intro} will follow from the following corollary of Theorem \ref{thm: ceresa section is multiple of branch point}. If $A$ is an abelian variety over $\Qbar$ of dimension $g$, then the Beilinson--Bloch height pairing \cite{Beilinson}
\[\langle \, , \, \rangle_{\mathrm{BB}} \colon \CH^i(A;\Q)_{\mathrm{hom}} \times \CH^{g+1-i}(A;\Q)_{\mathrm{hom}} \to \R\]
has been constructed unconditionally by K\"unnemann \cite[Corollary 1.7]{Kunnemann-heights}.

\begin{theorem}\label{thm: BB height = Neron-Tate height}
There is a rational number $N$ such that for all $a,b \in \Qbar$ with $\Delta_{a,b} \neq 0$, we have \[\langle \kappa(C_{a,b}), \kappa(C_{a,b}) \rangle_{\mathrm{BB}} = N^2 \langle \mathcal{Q}_{a,b}, \mathcal{Q}_{a,b} \rangle_{\mathrm{NT}},\]
where $\langle \, , \, \rangle_{\mathrm{NT}}$ is the N\'eron--Tate height pairing on $E^\Delta(\Qbar)$.
\end{theorem}

\begin{proof}
    This follows from Theorem \ref{thm: ceresa section is multiple of branch point} and Lemma \ref{lem: ceresa specializations}, using the fact that the Beilinson--Bloch height is compatible with correspondences \cite[4.0.3]{Beilinson} and agrees with the N\'eron--Tate height pairing for divisors on curves \cite[4.0.8]{Beilinson}.
\end{proof}

\begin{proof}[Proof of Theorem $\ref{thm: northcott intro}$]
    The theorem follows from Theorem \ref{thm: BB height = Neron-Tate height} and the nondegeneracy and Northcott properties of the N\'eron--Tate height of an elliptic curve. 
\end{proof}

\section{Algebraic triviality and the Griffiths group}\label{sec: Griffiths group}

Let $C$ be a curve over a field $k$ and $J$ be its Jacobian. The Griffiths group of $1$-cycles on $J$ is the group $\Gr_1(J) = \Gr^{g-1}(J) = \CH_1(J)_0/A_1(J)$, where $A_1(J)$ is the subgroup of algebraically trivial $1$-cycles. The very general Ceresa cycle (of a genus $g \geq 3$ curve over $\C$) is known to be of infinite order in $\Gr_1(J)$ \cite{Nori}. In particular, the Ceresa cycle of a very general curve is not algebraically trivial, nor is any multiple of it. In this section we show that this fails for bielliptic Picard curves. We also discuss the Beilinson--Bloch conjectures in this context, and show that they imply that $\Gr^2(J)$ is finite for certain bielliptic Picard curves.

\subsection{Proof of Theorem \ref{thm: torsion in griffiths intro}}

We keep the notation of \S\ref{subsec: an isomorphism of motives}. Theorem \ref{thm: torsion in griffiths intro} follows from the following proposition.
It essentially follows from the fact that homological and algebraic equivalence on a curve coincide.

\begin{proposition}
    Let $\kappa_{\Z}\in \CH^2(\mathcal{J})$ be a class and $N_1\geq 1$ an integer satisfying the conclusions of Lemma \ref{lemma: integral lift kappa_1}.
    Then there exists an integer $N\geq 1$, a smooth projective relative curve $\mathcal{X}\rightarrow S$ with two sections $\sigma_1, \sigma_2$ and a cycle $\mathcal{Z}$ on $\CH^1(\mathcal{J}\times_S X)$, with the property that $\mathcal{Z}_{\sigma_1} = N\cdot \kappa_{\Z}$ and $\mathcal{Z}_{\sigma_2} = 0$.
\end{proposition}

This proposition implies Theorem \ref{thm: torsion in griffiths intro}, since for every $k$-point $s\colon \Spec (k) \rightarrow S$, corresponding to the bielliptic Picard curve $\mathcal{C}_s = C$ with $X = \mathcal{X}_s$, $p_i = \sigma_i(s)$ and $Z = \mathcal{Z}_s$, the cycle $Z\in \CH^1(J\times_k X)$ has the property that $Z_{p_1} = s^*(\kappa_{\Z}) = NN_1\cdot \kappa_{\infty}(C)$ and $Z_{p_2} = 0$.
By definition, this means that $(NN_1)\cdot \kappa_{\infty}(C)$ is algebraically trivial.

\begin{proof}
We will use the isomorphisms $\Phi$ from Corollary \ref{cor: isomorphism of Chow groups}.
Recall from \S\ref{subsec: image of ceresa cycle} that there exists an integer $N_2\geq 1$ such that $\Phi_S(N_2\cdot \kappa(\mathcal{C})_{(1)}) = \sigma \in \mathcal{E}^{\Delta}(S)$.
We can deform the section $\sigma$ of $\mathcal{E}^{\Delta}$ to the identity section $\infty\colon S\rightarrow \mathcal{E}^{\Delta}$. 
More precisely, there exists a smooth projective relative curve $\mathcal{X}\rightarrow S$, a section $\mathcal{V} \in \mathcal{E}^{\Delta}(\mathcal{X})$ and sections $\sigma_1,\sigma_2$ of $\mathcal{X}\rightarrow S$, such that $\mathcal{V}_{\sigma_1} = \sigma$ and $\mathcal{V}_{\sigma_2} = \infty$.
In fact, $\mathcal{X} = \mathcal{E}^{\Delta}$ with $\sigma_1 = \sigma$, $\sigma_2 = \infty$ and with $\mathcal{V}=\Delta_{\mathcal{X}/S} \subset \mathcal{X}\times_S \mathcal{X} = \mathcal{E}^{\Delta}\times_S \mathcal{X}$ will do.
There exists an $N_3\geq 1$ such that $\Phi_{\mathcal{X}}^{-1}(N_3\cdot \mathcal{V})\in \CH^2_{(1)}(\mathcal{J}\times_S \mathcal{X};\Q)$ is the image of an integral cycle $\mathcal{W}_1 \in \CH^2(\mathcal{J}\times_S \mathcal{X})$.
Let $\mathcal{W} = N_1\cdot \mathcal{W}_1$.
The commutativity of the diagram of Corollary \ref{cor: isomorphism of Chow groups} applied to $\mathcal{X}\rightarrow S$ (which is flat, so admits pullbacks) shows that $\mathcal{W}_{\sigma_1} = N_1N_2N_3\cdot \kappa(\mathcal{C})_{(1)}=N_2N_3\kappa_{\Z}$ and $\mathcal{W}_{\sigma_2} = 0$ in $\CH^2(\mathcal{J};\Q)$.
In other words, the elements $\mathcal{W}_{\sigma_1}  - N_2N_3 \kappa_{\Z}$ and $\mathcal{W}_{\sigma_2}$ are torsion in $\CH^2(\mathcal{J})$. 
There exists an integer $N_4\geq 1$ such that these elements are killed by $N_4$, so the elements $N = N_2N_3N_4$ and $\mathcal{Z} = N_4\cdot \mathcal{W}$ satisfy the conclusion of the proposition.
\end{proof}

\subsection{Beilinson--Bloch conjectures}
We briefly describe some consequences of the Beilinson--Bloch conjecture \cite{Beilinson,BlochCrelleI} in our setting.  Let $X$ be a smooth projective geometrically integral variety over a number field $k$. 
Fix a prime $\ell$ and denote the $\ell$-adic cohomology $\HH^i(X_{\bar k}, \Q_{\ell})$ simply by $\HH^i(X)$.
The conjecture predicts that for each $i \in \Z$, the group of homologically trivial cycles $\CH^i(X)_{\mathrm{hom}}$ is finitely generated, the $L$-function $L(\frak{h}^{2i-1}(X), s)$ has analytic continuation and functional equation, and 
\[\mathrm{rk}\, (\CH^i(X)_{\mathrm{hom}}) = \ord_{s = i} L(\frak{h}^{2i-1}(X), s).\]
Similarly, if $\epsilon \in \CH^{\dim X}(X \times X)$ is an idempotent cutting out a direct summand of the Chow motive $\frak{h}(X)$, then  
\[\mathrm{rk}\, (\epsilon\circ \CH^i(X)_{\mathrm{hom}} )= \ord_{s = i} L(\epsilon\cdot  \frak{h}^{2i-1}(X), s).\]

Let $C = C_{a,b}$ be a bielliptic Picard curve over $k$ with Jacobian $J$. By Theorem \ref{thm: isomorphism of motives over a field} and Proposition \ref{prop: chow groups}, the Beilinson--Bloch conjecture for the motive $\frak{h}^3(J)^{\mu_6}$ reads:
\begin{conjecture}\label{conj: BB1}
 $\mathrm{rk} \, \CH^2_{(1)}(J)^{\mu_6} = \ord_{s = 1}L(\frak{h}^1(E^\Delta), s).$ 
\end{conjecture}

 The (weak) Birch and Swinnerton-Dyer conjecture states that $\ord_{s = 1}L(\frak{h}^1(E^\Delta), s) = \rk\, E^\Delta(k)$. 
Therefore the Beilinson--Bloch conjecture for $\frak{h}^3(J)^{\mu_6}$ is equivalent to the weak Birch and Swinnerton-Dyer conjecture for $E^\Delta$ (by Corollary \ref{cor: isomorphism of Chow groups}).     

\subsection{The Griffiths group}

As detailed in \cite[1.3]{BlochII}, the Beilinson--Bloch conjecture should also be compatible with the coniveau filtrations on $\CH^i(X)$ and $\HH^{2i-1}(X)(i)$. In our setting, this allows us to predict the rank of the Griffiths group $\Gr^2(J)$, at least in some cases; we briefly sketch the details below.
To formulate it, we use the decomposition $\frak{h}^2(P) = \frak{h}_{\mathrm{alg}}^2(P) + \frak{h}_{\mathrm{tr}}^2(P)$ from \cite{KahnMurrePedrini-transcendentalpartmotivesurface} which has the property that the $\ell$-adic realization of $\frak{h}^2_{{\mathrm{tr}}}(P)$ is \[V:= (\NS(P_{\bar{k}})(-1) \otimes \Q_\ell)^\perp \subset \HH^2(P),\]
the transcendental part of $\HH^2(P)$.

\begin{remark}
 {\em    If $P$ does not have CM then $V$ is three dimensional. If $P$ is moreover of $\GL_2$-type, then $V$ is (up to twist) the symmetric square of the associated 2-dimensional $\Gal_k$-representation. If $P$ has CM by an imaginary quadratic field $K$, then $V$ is $2$-dimensional and the $L$-function $L(V,s)$ is a product of Hecke $L$-functions attached to the field $K$.}
\end{remark}
Using the isogeny decomposition $J \sim P \times E$ and the Künneth formula, one checks that the first graded piece $\mathrm{gr}^0\HH^3(J)(2)$ in the coniveau filtration is a nonzero quotient of $V \otimes \HH^1(E)(2)$. Define 
\[\frak{h}_{\mathrm{Gr}}^3(J) := \frak{h}_{{\mathrm{tr}}}^2(P) \otimes \frak{h}^1(E).\] 
The following prediction then follows from \cite[1.3]{BlochII}.

\begin{conjecture}
Let $C$ be a bielliptic Picard curve over a number field $k$. Then \[\mathrm{rk} \, \Gr^2(J) \leq \ord_{s = 2}L(\frak{h}_{\mathrm{Gr}}^3(J),s),\]
with equality if $V \otimes \HH^1(E)$ is an irreducible $\Gal_k$-representation.   
\end{conjecture}

Assuming this conjecture, we exhibit  below some new examples of nonhyperelliptic Jacobians over a number fields with finite Griffiths group. First we describe some known examples.

\begin{example}
{\em 
    Following up on \cite{BisognoLiLittSrinivasan}, Gross found two nonhyperelliptic Jacobians with finite Griffiths group (assuming the Beilinson--Bloch conjecture): the genus $7$ Fricke-Macbeath curve and its genus $3$ quotient \cite{GrossCeresa}, both defined over $\Q$.   
    }
\end{example}
The other previously known example is a certain ``special'' bielliptic Picard curve. To put it in context, note that if $P$ is defined over $\Q$ and has CM (over $\Qbar$) by an imaginary quadratic field $\Q(\sqrt{d}) \neq \Q(\omega)$, then $L(\frak{h}_{\mathrm{Gr}}^3(J),s)$ is a Hecke $L$-function for the biquadratic CM field $\Q(\omega, \sqrt{d})$.   

\begin{example}\label{ex: LS special curve}
{\em 
The curve $C_{0,1} \colon y^3 = x^4 + 1$ has CM by $\Q(i)$ \cite[Example 6.15]{LagaShnidman}. In \cite[Theorem 1.3]{LilienfeldtShnidman}, Lilienfeldt and the second author show that the corresponding Hecke $L$-function has nonvanishing central value, and hence $\Gr^2(J_{0,1})$ is conjecturally finite.    
}
\end{example}

Using Theorem \ref{thm: Neron Severi action} and mod $p$ point counts of $E$ and $J$, one can compute the good Euler factors for the $L$-function $L(\frak{h}^3_{\mathrm{Gr}}(J),s)$, even for non-CM bielliptic Picard curves $C_{a,b}$. These $L$-functions are not known to have analytic continuation in general but we can still compute their analytic rank in Magma assuming that they do.

\begin{example}\label{ex: griff}
{\em 
Using \cite{AsifFitePentland, Sutherland2020}, we have verified (conditional on the analytic continuation of the $L$-function) that for both $C_{6,-3}$ and $C_{6,1}$, the central value $L(\frak{h}^3_{\mathrm{Gr}}(J),2)$ is nonzero. Hence $\Gr^2(J_{6,-3})$ and $\Gr^2(J_{6,1})$ are conjecturally finite groups.}    
\end{example}

Theorem \ref{thm: torsion in griffiths intro} suggests that the behavior in Examples \ref{ex: LS special curve} and \ref{ex: griff} is not so special, and we expect that many bielliptic Picard Jacobians over $\Q$ have finite Griffiths group.

\bibliographystyle{abbrv}

\begin{thebibliography}{10}

\bibitem{SGA5}
{\em Cohomologie {$l$}-adique et fonctions {$L$}}, volume Vol. 589 of {\em
  Lecture Notes in Mathematics}.
\newblock Springer-Verlag, Berlin-New York, 1977.
\newblock S\'{e}minaire de G\'{e}ometrie Alg\'{e}brique du Bois-Marie
  1965--1966 (SGA 5), Edit\'{e} par Luc Illusie.

\bibitem{AsifFitePentland}
S.~Asif, F.~Fit\'{e}, and D.~Pentland.
\newblock Computing {$L$}-polynomials of {P}icard curves from {C}artier-{M}anin
  matrices.
\newblock {\em Math. Comp.}, 91(334):943--971, 2022.
\newblock With an appendix by A. V. Sutherland.

\bibitem{Beauville-surlanneadechow}
A.~Beauville.
\newblock Sur l'anneau de {C}how d'une vari\'{e}t\'{e} ab\'{e}lienne.
\newblock {\em Math. Ann.}, 273(4):647--651, 1986.

\bibitem{Beauville}
A.~Beauville.
\newblock A non-hyperelliptic curve with torsion {C}eresa class.
\newblock {\em C. R. Math. Acad. Sci. Paris}, 359:871--872, 2021.

\bibitem{BeauvilleSchoen}
A.~Beauville and C.~Schoen.
\newblock A non-hyperelliptic curve with torsion {C}eresa cycle modulo
  algebraic equivalence.
\newblock {\em Int. Math. Res. Not. IMRN}, (5):3671--3675, 2023.

\bibitem{Beilinson}
A.~A. Be\u{\i}linson.
\newblock Height pairing between algebraic cycles.
\newblock In {\em {$K$}-theory, arithmetic and geometry ({M}oscow,
  1984--1986)}, volume 1289 of {\em Lecture Notes in Math.}, pages 1--25.
  Springer, Berlin, 1987.

\bibitem{BisognoLiLittSrinivasan}
D.~Bisogno, W.~Li, D.~Litt, and P.~Srinivasan.
\newblock Group-theoretic {J}ohnson classes and non-hyperelliptic curves with
  torsion {C}eresa class.
\newblock {\em \'{E}pijournal G\'{e}om. Alg\'{e}brique}, 7:Art. 8, 19, 2023.

\bibitem{BlochCrelleI}
S.~Bloch.
\newblock Algebraic cycles and values of {$L$}-functions.
\newblock {\em J. Reine Angew. Math.}, 350:94--108, 1984.

\bibitem{BlochII}
S.~Bloch.
\newblock Algebraic cycles and values of {$L$}-functions. {II}.
\newblock {\em Duke Math. J.}, 52(2):379--397, 1985.

\bibitem{BLR-neronmodels}
S.~Bosch, W.~L\"{u}tkebohmert, and M.~Raynaud.
\newblock {\em N\'{e}ron models}, volume~21 of {\em Ergebnisse der Mathematik
  und ihrer Grenzgebiete (3) [Results in Mathematics and Related Areas (3)]}.
\newblock Springer-Verlag, Berlin, 1990.

\bibitem{BuhlerSchoenTop}
J.~Buhler, C.~Schoen, and J.~Top.
\newblock Cycles, {$L$}-functions and triple products of elliptic curves.
\newblock {\em J. Reine Angew. Math.}, 492:93--133, 1997.

\bibitem{Ceresa}
G.~Ceresa.
\newblock {$C$} is not algebraically equivalent to {$C\sp{-}$} in its
  {J}acobian.
\newblock {\em Ann. of Math. (2)}, 117(2):285--291, 1983.

\bibitem{deJong}
R.~de~Jong.
\newblock On the height of {G}ross-{S}choen cycles in genus three.
\newblock {\em Res. Number Theory}, 4(4):Paper No. 38, 25, 2018.

\bibitem{dejong2021jumps}
R.~de~Jong and F.~Shokrieh.
\newblock Jumps in the height of the ceresa cycle, 2021.
\newblock Preprint, available at \url{https://arxiv.org/abs/2104.10060}.

\bibitem{DeningerMurre}
C.~Deninger and J.~Murre.
\newblock Motivic decomposition of abelian schemes and the {F}ourier transform.
\newblock {\em J. Reine Angew. Math.}, 422:201--219, 1991.

\bibitem{Donagi-tetragonalconstruction}
R.~Donagi.
\newblock The tetragonal construction.
\newblock {\em Bull. Amer. Math. Soc. (N.S.)}, 4(2):181--185, 1981.

\bibitem{EllenbergLoganSrinivasan}
J.~Ellenberg, A.~Logan, and P.~Srinivasan.
\newblock Certifying nontriviality of ceresa classes of curves.
\newblock Preprint, available at \url{https://www.arxiv.org/abs/2412.02015v1},
  2024.

\bibitem{EskandariMurty}
P.~Eskandari and V.~K. Murty.
\newblock On the harmonic volume of {F}ermat curves.
\newblock {\em Proc. Amer. Math. Soc.}, 149(5):1919--1928, 2021.

\bibitem{Faltings-finitenesstheorems}
G.~Faltings.
\newblock Finiteness theorems for abelian varieties over number fields.
\newblock In {\em Arithmetic geometry ({S}torrs, {C}onn., 1984)}, pages 9--27.
  Springer, New York, 1986.
\newblock Translated from the German original [Invent. Math. {\bf 73} (1983),
  no. 3, 349--366; MR0718935; ibid. {\bf 75} (1984), no. 2, 381; MR0732554] by
  Edward Shipz.

\bibitem{faltingschai}
G.~Faltings and C.-L. Chai.
\newblock {\em Degeneration of abelian varieties}, volume~22 of {\em Ergebnisse
  der Mathematik und ihrer Grenzgebiete (3) [Results in Mathematics and Related
  Areas (3)]}.
\newblock Springer-Verlag, Berlin, 1990.
\newblock With an appendix by David Mumford.

\bibitem{Fulton-intersectiontheory}
W.~Fulton.
\newblock {\em Intersection theory}, volume~2 of {\em Ergebnisse der Mathematik
  und ihrer Grenzgebiete. 3. Folge. A Series of Modern Surveys in Mathematics
  [Results in Mathematics and Related Areas. 3rd Series. A Series of Modern
  Surveys in Mathematics]}.
\newblock Springer-Verlag, Berlin, second edition, 1998.

\bibitem{GaoZhang-NorthcottCeresa}
Z.~Gao and S.-W. Zhang.
\newblock Heights and periods of algebraic cycles in families.
\newblock Arxiv preprint, available at
  \url{https://arxiv.org/abs/2407.01304v1}, 2024+.

\bibitem{GrossCeresa}
B.~H. Gross.
\newblock The {F}ricke–{M}acbeath curve and triple product {L}-functions.
\newblock Talk, slides available at
  \url{https://www.fields.utoronto.ca/talk-media/1/40/60/slides.pdf}, 2021.

\bibitem{GrossSchoen}
B.~H. Gross and C.~Schoen.
\newblock The modified diagonal cycle on the triple product of a pointed curve.
\newblock {\em Ann. Inst. Fourier (Grenoble)}, 45(3):649--679, 1995.

\bibitem{Jannsen}
U.~Jannsen.
\newblock {\em Mixed motives and algebraic {$K$}-theory}, volume 1400 of {\em
  Lecture Notes in Mathematics}.
\newblock Springer-Verlag, Berlin, 1990.
\newblock With appendices by S. Bloch and C. Schoen.

\bibitem{KahnMurrePedrini-transcendentalpartmotivesurface}
B.~Kahn, J.~P. Murre, and C.~Pedrini.
\newblock On the transcendental part of the motive of a surface.
\newblock In {\em Algebraic cycles and motives. {V}ol. 2}, volume 344 of {\em
  London Math. Soc. Lecture Note Ser.}, pages 143--202. Cambridge Univ. Press,
  Cambridge, 2007.

\bibitem{KerrTayou-CeresaHodgetheory}
M.~Kerr and S.~Tayou.
\newblock On the torsion locus of the ceresa normal function.
\newblock Arxiv preprint, available at
  \url{https://arxiv.org/abs/2406.19366v1}, 2024+.

\bibitem{Kings-higherregulators}
G.~Kings.
\newblock Higher regulators, {H}ilbert modular surfaces, and special values of
  {$L$}-functions.
\newblock {\em Duke Math. J.}, 92(1):61--127, 1998.

\bibitem{Kunnemann-arakelovchowgroups}
K.~K\"{u}nnemann.
\newblock Arakelov {C}how groups of abelian schemes, arithmetic {F}ourier
  transform, and analogues of the standard conjectures of {L}efschetz type.
\newblock {\em Math. Ann.}, 300(3):365--392, 1994.

\bibitem{kunnemann}
K.~K\"{u}nnemann.
\newblock On the {C}how motive of an abelian scheme.
\newblock In {\em Motives ({S}eattle, {WA}, 1991)}, volume 55, Part 1 of {\em
  Proc. Sympos. Pure Math.}, pages 189--205. Amer. Math. Soc., Providence, RI,
  1994.

\bibitem{Kunnemann-heights}
K.~K\"{u}nnemann.
\newblock Height pairings for algebraic cycles on abelian varieties.
\newblock {\em Ann. Sci. \'{E}cole Norm. Sup. (4)}, 34(4):503--523, 2001.

\bibitem{LagaShnidman}
J.~Laga and A.~Shnidman.
\newblock The geometry and arithmetic of bielliptic {P}icard curves.
\newblock Arxiv preprint, available at
  \url{https://arxiv.org/abs/2308.15297v2}, 2023+.

\bibitem{LagaShnidman-VanishingcriteriaCeresa}
J.~Laga and A.~Shnidman.
\newblock Vanishing criteria for {C}eresa cycles.
\newblock Arxiv preprint, available at
  \url{https://arxiv.org/abs/2406.03891v1}, 2024+.

\bibitem{Laterveer}
R.~Laterveer.
\newblock On the tautological ring of {H}umbert curves.
\newblock {\em manuscripta math.}, 172:1093--1107, 2023.

\bibitem{LilienfeldtShnidman}
D.~T.-B.~G. Lilienfeldt and A.~Shnidman.
\newblock Experiments with {C}eresa classes of cyclic {F}ermat quotients.
\newblock {\em Proc. Amer. Math. Soc.}, 151(3):931--947, 2023.

\bibitem{MazurRubinSilverberg}
B.~Mazur, K.~Rubin, and A.~Silverberg.
\newblock Twisting commutative algebraic groups.
\newblock {\em J. Algebra}, 314(1):419--438, 2007.

\bibitem{Moonen}
B.~Moonen.
\newblock On the {C}how motive of an abelian scheme with non-trivial
  endomorphisms.
\newblock {\em J. Reine Angew. Math.}, 711:75--109, 2016.

\bibitem{MunozPorras-abeljacobi}
J.~M. Mu\~{n}oz Porras.
\newblock On the structure of the birational {A}bel morphism.
\newblock {\em Math. Ann.}, 281(1):1--6, 1988.

\bibitem{MurreNagelPeters}
J.~P. Murre, J.~Nagel, and C.~A.~M. Peters.
\newblock {\em Lectures on the theory of pure motives}, volume~61 of {\em
  University Lecture Series}.
\newblock American Mathematical Society, Providence, RI, 2013.

\bibitem{Nori}
M.~V. Nori.
\newblock Cycles on the generic abelian threefold.
\newblock {\em Proc. Indian Acad. Sci. Math. Sci.}, 99(3):191--196, 1989.

\bibitem{Pantazis-Prymvarsgeodesicflow}
S.~Pantazis.
\newblock Prym varieties and the geodesic flow on {${\rm SO}(n)$}.
\newblock {\em Math. Ann.}, 273(2):297--315, 1986.

\bibitem{QiuZhang}
C.~Qiu and W.~Zhang.
\newblock Vanishing results in {C}how groups for the modified diagonal cycles.
\newblock Preprint, available at \url{https://arxiv.org/abs/2209.09736v3},
  2022.

\bibitem{QiuZhangII}
C.~Qiu and W.~Zhang.
\newblock Vanishing results in {C}how groups for the modified diagonal cycles
  {II}: {S}himura curves.
\newblock Preprint, available at \url{https://arxiv.org/abs/2310.19707v1},
  2023.

\bibitem{Scholl-classicalmotives}
A.~J. Scholl.
\newblock Classical motives.
\newblock In {\em Motives ({S}eattle, {WA}, 1991)}, volume~55 of {\em Proc.
  Sympos. Pure Math.}, pages 163--187. Amer. Math. Soc., Providence, RI, 1994.

\bibitem{SchuttShioda-ellipticsurfaces}
M.~Sch\"{u}tt and T.~Shioda.
\newblock Elliptic surfaces.
\newblock In {\em Algebraic geometry in {E}ast {A}sia---{S}eoul 2008},
  volume~60 of {\em Adv. Stud. Pure Math.}, pages 51--160. Math. Soc. Japan,
  Tokyo, 2010.

\bibitem{stacksproject}
T.~{Stacks Project Authors}.
\newblock \textit{Stacks Project}, 2018.
\newblock \url{https://stacks.math.columbia.edu}.

\bibitem{Sutherland2020}
A.~V. Sutherland.
\newblock Counting points on superelliptic curves in average polynomial time.
\newblock In {\em A{NTS} {XIV}---{P}roceedings of the {F}ourteenth
  {A}lgorithmic {N}umber {T}heory {S}ymposium}, volume~4 of {\em Open Book
  Ser.}, pages 403--422. Math. Sci. Publ., Berkeley, CA, 2020.

\bibitem{Zelinsky}
D.~S. Zelinsky.
\newblock Some abelian threefolds with nontrivial {G}riffiths group.
\newblock {\em Compositio Math.}, 78(3):315--355, 1991.

\bibitem{ShouwuZhang}
S.-W. Zhang.
\newblock Gross-{S}choen cycles and dualising sheaves.
\newblock {\em Invent. Math.}, 179(1):1--73, 2010.

\end{thebibliography}

\end{document}